\documentclass[10pt,a4paper]{amsart}
\usepackage[utf8]{inputenc}
\usepackage[T1]{fontenc}
\usepackage{amsmath}
\usepackage{amsthm}
\usepackage{amsfonts}
\usepackage{amssymb}
\usepackage{cleveref}
\usepackage[shortlabels]{enumitem}
\usepackage{enumitem}
\usepackage{graphicx}
\usepackage{mathtools}
\usepackage{ragged2e}
\usepackage{resmes}

\DeclareMathOperator*{\ba}{ba}
\DeclareMathOperator*{\cl}{cl}

\DeclareMathOperator*{\dist}{dist}
\DeclareMathOperator*{\DIV}{div}

\DeclareMathOperator*{\ext}{ext}
\DeclareMathOperator*{\interior}{int}

\DeclareMathOperator*{\ran}{ran}

\DeclareMathOperator*{\supp}{supp}

\def\e{\varepsilon}
\def\i{\infty}
\def\p{\partial}
\def\N{{\mathbb N}}

\def\R{{\mathbb R}}
\def\AA{{\mathcal A}}
\def\BB{{\mathcal B}}
\def\BD{{\mathcal {BD}}}
\def\BF{{\mathcal {BF}^\i}}
\def\BV{{\mathcal {BV}}}
\def\DD{{\mathcal D}}
\def\DM{{\mathcal {BDM}}}
\def\DMN{{\mathcal {BDM}}_{0}}

\def\HH{{\mathcal H}}
\def\MM{{\mathcal M}}
\def\NN{{\mathcal N}}

\def\WW{{\mathcal W}}

\def\weakast{\overset{\ast}{\rightharpoonup}}

\newtheorem{corollary}{Corollary}[section]
\newtheorem{definition}[corollary]{Definition}
\newtheorem{lemma}[corollary]{Lemma}
\newtheorem{proposition}[corollary]{Proposition}

\newtheorem{theorem}[corollary]{Theorem}

\def\XXint#1#2#3{{\setbox0=\hbox{$#1{#2#3}{\int}$} 
		\vcenter{\hbox{$#2#3$}}\kern-.5\wd0}}

\title{A rough divergence theorem}

\author{Thomas Ruf}
\address[Thomas Ruf]{Institut für Geometrie, TU Dresden, 01069 Dresden, Germany}
\email{thomas.ruf@tu-dresden.de, thomas.ruf.math@web.de}

\begin{document}
	
	\begin{abstract}
		A generalized divergence theorem is established allowing for domains with inner boundaries. The normal trace of a rough integrand is not a Radon measure; rather, the boundary integral is expressed via a surface functional continuous with respect to the uniform convergence of integrands. We provide necessary and sufficient analytic and geometric conditions on the domain for the validity of the theorem. Central to this characterization is the introduction of the space of functions having bounded fluctuation, whose norm is precisely defined so that the divergence theorem holds if and only if the characteristic function $\chi_U$ of the integration domain $U \subset \R^m$ has finite norm.
	\end{abstract}
	
	\maketitle
	
	\section{Introduction}
	
	Many extensions of the divergence theorem admit rough data, either in the integrand or in the domain, cf. the survey \cite{CT2}. Typically, the integration is taken over the measure-theoretic interior, thus removing inner boundaries. Hence, such extensions fail even on domains that possess a simple inner structure. Here, a formulation is given that includes such domains $U \subset \R^m$ whenever the non-exterior boundary has finite 1-codimensional Hausdorff measure, i.e.,
	\begin{equation} \label{eq:pres reg}
		\HH^{m - 1}(\p U \setminus \mathrm{ext}_* U) < \i,
	\end{equation}
	where $\HH^{m - 1}$ is Hausdorff measure, $\p U$ is the topological boundary and $\ext_* U$ the measure theoretic exterior.	The divergence integral is interpreted as a functional on a space of bounded, rough vector fields, whose dual then provides a surface functional that is continuous under uniform convergence. It is concentrated on the boundary in a suitable sense. We also show that \eqref{eq:pres reg} is necessary. It is equivalent to the smooth approximability of the domain and of its indicator function, the approximations proceeding from inside or from below with uniformly bounded variation. The equivalence of these statements forms the main result. In simplified form:
	
	\begin{theorem}
		Let $U \subset \R^m$ have finite measure. The following are equivalent:
		\begin{enumerate}[(i)]
			
			\item $U$ admits a divergence theorem with a surface functional as above;
			
			\item There holds \eqref{eq:pres reg};			 
			
			\item $U$ or $\chi_U$ can be approximated from inside or below by smooth sets or functions with uniformly $L^1$-bounded gradients.
			
		\end{enumerate}
	\end{theorem}
	Remarkably, the condition on the non-exterior boundary does not entail rectifiability. It is analogous to the reduction of the boundary for sets of finite perimeter, except that only the outer part is reduced.
	\\
	
	\emph{Strategy of proof.} The main analytic device for formulating and analyzing our divergence theorem is a seminorm inspired by the $\BV$-seminorm, where smooth test fields are replaced with the space $\DM$ of bounded fields whose divergence is a Radon measure and which need not vanish near the boundary, i.e.,
	$$
		\| f \|_{BF(U) } \coloneqq \sup \left\{ \int_U f \, \mathrm{d} \DIV u \colon u \in \DM(U), \, \| u \|_{L^\i} \le 1 \right\}.
	$$
	This bounded fluctuation or $\BF$-seminorm is finite for a domain indicator $\chi_U$ if and only if the domain satisfies our divergence theorem. The space $\BF$ provides a functional framework for studying approximability: by showing that test functions are dense in a weak topology of $\BF$, we obtain the necessity of approximability for $\chi_U$. The transition from smooth functions to smooth sets follows from the coarea formula, the Sard lemma, and convolution smoothing. The equivalence between the Hausdorff measure condition on the boundary and smooth set approximability rests on \cite[Thm. 3.1]{CLT}, whose argument is extended here to the present framework. The sufficiency of smooth set approximation for the divergence theorem reduces to lower semicontinuity of the $\BF$-seminorm. Since this property is not immediate, we replace the $\BF$-seminorm by a related one whose defining supremum is manifestly lower semicontinuous under strong $L^1$-convergence. The related seminorm is then shown to agree with the $\BF$-seminorm for sufficiently regular functions such as $\chi_U$.
	\\
	
	\emph{Relation to earlier results.}
	Divergence theorems for inner boundaries occur in \cite[Thm. 3.18]{SS}, \cite[Thm. 5.2]{CLT}, \cite[Prop. 4.34(2)]{SS3}. The first assumes stronger regularity on $\p U$ and is equivalent in form to the present divergence theorem; the second first considered \eqref{eq:pres reg} as a sufficient condition in this context, permitting the formally weaker condition $u \in \DM(U) \implies \chi_U u \in DM(\R^m)$. However, it omitted independence of the surface functional from the integrand, constructing a Radon regular normal trace functional for each element of $\DM(U)$; The third assumes only \eqref{eq:pres reg} on the domain and a linear bound for the normal trace functional in terms of the uniform norm of $u \in \DM(U)$ was established. However, by contrast to \cite[Thm. 3.18]{SS}, the surface functional is no longer shown to be independent of the integrand vector field.
	\\
	The $\BF$-seminorm derives from the $\BV$-seminorm as first introduced in \cite{Mi} by admission of test fields in $\DM(U)$. The resulting functional interpretation of a surface term replaces the distributional trace by a continuous linear functional on a space of rough functions. In this respect, the theory generalizes some of Anzellotti pairing theory: if $f$ has finite bounded fluctuation, and $\BD(U)$ denotes the closure of $\DM(U)$ in $L^\i(U)$, then there is a unique functional $\ell_f \in \BD'(U)$ such that
	$$
	\int_U f \, \mathrm{d} \DIV u = \langle \ell_f, u \rangle \quad \forall u \in \DM(U).
	$$
	Thus, the distribution
	$$
	C^\i_c(U) \to \R \colon \varphi \mapsto \langle \ell_f, \varphi u \rangle
	$$
	coincides with the pairing $(u, Df)$ of \cite{CC}. Conversely, the functional $\ell_f$ cannot be recovered from $(u, Df)$ since $\ell_f$ takes the boundary behavior of $f$ into account whereas any $\varphi \in C^\i_c(U)$ vanishes near the boundary. Accordingly, the present paper does not rely on Anzellotti pairing theory, rather presenting an extension.
	\\
	
	\emph{Future directions.} Possible extensions include a formulation of the divergence theorem for domains with fractal or highly irregular boundaries, obtained by replacing the $L^\i$-norm in the $\BF$-seminorm with a fractional-order counterpart. Conversely, one may ask for the broadest class of integrand functions admitting a surface functional continuous under uniform convergence for domains of prescribed regularity, such as of finite perimeter.
	\\
	A general unbounded integrand may admit a distributional normal trace below measure regularity, even on a regular domain, cf. \cite[Ex. 2.5]{Sh3}. The present divergence theorem furnishes a framework to identify those unbounded integrands for which the surface functional can be defined.
	\\
	The weak derivative concept underlying $\BF$ generalizes the framework of Anzellotti pairings; its systematic study could yield simplified calculus rules and fine properties. In particular, the chain rule $D(F \circ u) = F'(u) \, Du$ could become true and arise with less technical effort than its classical counterpart in $\BV$, where the jump part behaves exceptionally, see \cite[Thm. 3.96]{AFP}.
	\\
	
	\emph{Structure of the paper.} \Cref{sec:not}: Notations and preliminaries. \Cref{sec:BF}: Definition of the bounded fluctuation seminorm and space; elementary properties; denseness of test functions; coincidence with an auxiliary seminorm under a sufficient condition. \Cref{sec:divthm}: Proofs of the nontrivial implications of the main theorem. Appendix: Proofs of a direct-sum decomposition for the dual of an $L^\i$-type space and of a generalized Goldstine theorem.
	
	\section{Notation and preliminary material} \label{sec:not}
	
	In addition to standard notation from analysis and measure theory, the following conventions are used throughout. Fix $m \in \N$ and let $U \subset \R^m$ be open. The $\sigma$-algebra of Borel subsets of $U$ is denoted by $\BB(U)$; its completion with respect to the restricted $(m - 1)$-dimensional Hausdorff measure $\HH^{m-1} \resmes U$ is $\bar{\BB}(U)$.
	
	The paper employs the usual function spaces: $L^p(U)$ for $1 \le p \le \i$; $C(U)$; $C^k(U)$ and $C^\i(U)$ for $k \in \N$; their subspaces of compactly supported functions $C_c(U)$, $C^k_c(U)$, $C^\i_c(U)$; the closure of $C_c(U)$ is $C_0(U)$; the space of continuous bounded functions $C_b(U)$; signed Radon measures $\MM(U)$; and functions of bounded variation $\mathcal{BV}(U)$, all taken with their standard Banach space norms.
	
	Of particular interest is the space of bounded divergence-measure fields,
	$$
	\DM(U) \coloneqq \left\{ u \in L^\i(U; \R^m) \colon \DIV u \in \MM(U) \right\},
	$$
	along with the spaces
	\begin{alignat*}{2}
		\DM_c(U) 	& \coloneq \left\{ u \in \DM(U) \colon \supp(u) \subset \subset U \right\}, \\[0.3em]
		\DM_b(U) 	& \coloneq \left\{ u \in \DM(U) \colon \supp(u) \text{ is bounded} \right\}, \\[0.3em]
		\BD(U)      & \coloneq \overline{\DM(U) } \text{ with closure in } \| \cdot \|_\i, \\[0.3em]
		\BD_0(U)    & \coloneq \overline{\DM_c(U) }  \text{ with closure in } \| \cdot \|_\i, \\[0.3em]
		\BD_b(U)    & \coloneq \overline{\DM_b(U) }  \text{ with closure in } \| \cdot \|_\i, \\[0.3em]
		\MM_{\DIV}(U) & \coloneq \left\{ \mu \in \MM(U) \colon \exists u \in \DM(U) \text{ with } \DIV u = \mu \right\}.
	\end{alignat*}
	These are equipped with the essential supremum norm, under which they form Banach spaces, except for $\MM_{\DIV}(U)$, which takes to total variation norm.
	
	If $u \in \DM(U)$, then $\DIV u \ll \HH^{m-1} \resmes U$ \cite[Cor. 2.3.1]{Co}. Any real-valued, $\bar{\BB}(U)$-measurable function coincides $\HH^{m - 1} \resmes U$-almost everywhere with a $\BB(U)$-measurable function by \cite[Lem. A.3]{Ru}. Thus, every bounded, $\bar{\BB}(U)$-measurable function is integrable with respect to $\DIV u$ for all $u \in \DM(U)$. This motivates the preference for $\bar{\BB}(U)$ over stricter notions of measurability. Typically, the role of $\bar{\BB}(U)$ is played by the collection of Lebesgue-measurable sets in similar contexts; however, Lebesgue null sets may be too large to permit modification of a function without affecting integrals against divergence measures, unlike $\HH^{m - 1} \resmes U$-null sets.
	
	For a set $E \subset \R^m$, the measure-theoretic interior, exterior and boundary are denoted by $\mathrm{\interior}_* E$, $\mathrm{\ext}_* E$, and $\p_* E$ respectively, and the reduced boundary by $\p^* E$. The perimeter of $E$ in $U$ is written $P(E; U)$, or simply $P(E)$ if $U = \R^m$. For details on sets of finite perimeter, see \cite{AFP, EG, Fe3, Ma}.
	
	Given locally convex (real) vector spaces $X, Y$ in duality via $\langle \cdot, \cdot \rangle$, $\sigma(X, Y)$ is the weak topology on $X$ induced by $Y$. For $A \subset X$, the indicator function in the sense of convex analysis is $I_A(x) = +\i$ for $x \in A$, zero otherwise. The support function of $S \subset X$ is $s_A(y) = s(A, y) = \sup_{x \in S} \langle x, y \rangle$ for $y \in Y$. For a function $f : X \to [ - \i, + \i ]$, the convex conjugate is $f^*(y) = \sup_{x \in X} \langle x, y \rangle - f(x)$, and $I_A^* = s_A$. Further background in convex analysis may be found in \cite[§3]{IT}.
	
	I record a simple lemma regarding the approximation of a bounded divergence measure field by smooth ones. The proof is omitted since it uses standard arguments involving convolution smoothing and smooth partitions of unity as in the analogous approximation results for Sobolev and bounded variation functions \cite[Thm. 4.2, Thm. 5.3]{EG}.
	
	\begin{lemma} \label{lem:MSinDM}
		For every $u \in \DM(U)$, there is a sequence of functions $u_k \in \DM(U) \cap C^\i(U; \R^m)$ with $\DIV u_k \in L^1(U) \cap C_b(U)$ and $\| u_k \|_\i \le \| u \|_\i$ for all $k \in \N$ such that $u_k$ converges to $u$ weakly* in $L^\i(U)$, $\DIV u_k$ converges to $\DIV u$ weakly* in $\MM(U)$, and $\| \DIV u_k \|(U)$ converges to $\| \DIV u \|(U)$. If moreover $u \in \DM(U) \cap C_b(U)$, then $u_k$ converges to $u$ strongly in $C_b(U)$.
	\end{lemma}
	
	For reference, we record one last simple observation.
	
	\begin{lemma} \label{lem:div rich}
		Let $\WW_0(U)$ be the homogeneous Sobolev space arising as the closure of test functions in the Poincaré norm $\| \nabla \cdot \|_{L^1(U) }$. The map $- \nabla \colon\WW_0(U) \to L^1(U)$ is isometric, hence the adjoint operator $\DIV \colon L^\i(U) \to \WW_0'(U)$ is a quotient map. In particular, by the Sobolev embedding theorem, its image contains all compactly supported, bounded, measurable functions.
	\end{lemma}
	
	\section{Functions of bounded fluctuation} \label{sec:BF}
	
	In this section, I define bounded fluctuation functions and study their basic properties. My terminology emphasizes analogies with bounded variation functions wherever possible. Throughout this section, $U \subset \R^m$ is an open set.
	
	\subsection{Definition and structure theory} \label{ssec:BF1}
	
	\begin{definition} \label{def:bndfluc}
		A bounded, $\bar{\BB}(U)$-measurable function $f \colon U \to \R$ has \emph{bounded fluctuation} in $U$ if for every $u \in \DM(U)$, the \emph{bounded fluctuation seminorm}
		\begin{equation} \label{eq:bfseminorm}
			 \| f \|_{BF(U) } \coloneqq \sup \left\{ \int_U f \, \mathrm{d} \DIV u \colon u \in \DM(U), \, \| u \|_{L^\i(U; \, \R^m) } \le 1 \right\}
		\end{equation}
		is finite. The vector space of such functions $f$ is denoted by
		$$
		\BF(U).
		$$
		A Borel measurable set $E \subset U$ has \emph{finite anexometer} in $U$ if the indicator function $\chi_E$ has bounded fluctuation in $U$.	If $U$ has finite anexometer in itself, we simply say that $U$ has finite anexometer.
	\end{definition}
	
	\emph{Remarks}:
	
	\begin{enumerate}[(i)]
		
		\item Short-hand notations like $\| \cdot \|_{BF}$ instead of $\| \cdot \|_{BF(U) }$ are used if the underlying set is clear from context.
		
		\item The bounded fluctuation seminorm being a supremum of linear functions, it is sublinear. The set of functions being a vector space, $\| \cdot \|_{BF}$ is even and non-negative, rendering it a seminorm on the vector space where it is finite.
		
		\item Obviously, the BV-seminorm $\| D f \|(U)$ is dominated by $\| f \|_{BF(U) }$. In particular, every $f \in \BF(U)$ has a distributional gradient in $\MM(U)$. A refinement relating to the distributional derivative of bounded fluctuation functions will be proved in \Cref{lem:muf}. The converse inequality of seminorms may fail. Indeed, below \Cref{thm:main}, I present an example of a bounded set having infinite anexometer but finite perimeter.
		\\
		
	\end{enumerate}
	
	If a divergence theorem of the form
	$$
	\exists \mu_{\chi_U} \in \BD'(U) \colon \int_U \mathrm{d} \DIV u = \langle \mu_{\chi_U}, u \rangle \quad \forall u \in \DM(U)
	$$	
	is to hold, then the set $U$ necessarily has finite anexometer. The same condition suffices by \Cref{lem:muf} below. More generally, \Cref{lem:muf} tells us that a function of bounded fluctuation not only has a distributional Radon measure derivative, but also enjoys a well-defined boundary behavior encoded in a generalized derivative possessing a unique additive decomposition where the boundary behavior is isolated in a single addend of the derivative. This is analogous to the classical integration by parts formula
	\begin{equation} \label{eq:ibpclass}
		\int_U f \DIV g \, \mathrm{d}x = \int_{\p U} f g \, \mathrm{d} \HH^{m - 1} - \int_u \nabla f \cdot g \, \mathrm{d}x
	\end{equation}
	for a smooth domain $U$ and smooth functions $f \colon \R \to \R$, $g \colon \R^m \to \R^m$ if we regard the trace $f \, \mathrm{d} \HH^{m - 1} \resmes \p U$ as pertinent to the derivative.
	
	\begin{lemma} \label{lem:muf}
		For every $f \in \BF(U)$, there is a unique, continuous, linear functional $\mu_f \in \BD'(U)$ satisfying
		\begin{equation} \label{eq:mufgauß}
			- \int_U f \, \mathrm{d} \DIV u = \langle \mu_f, u \rangle \quad \forall u \in \DM(U), \quad \| f \|_{BF} = \| \mu_f \|.
		\end{equation}
		In particular, decomposing $\mu_f = D f + B f + I f$ with $Df \in \BD'_0(U)$, $Bf \in \BD'_{\p U}(U)$, and $I f \in \BD'_\i(U)$ according to \Cref{cor:DBdeco}, there holds
		\begin{equation} \label{eq:mufnorm}
			\| f \|_{BF} = \| Df \| + \| Bf \| + \| I f \|.
		\end{equation}
		Moreover, the restriction of $Df \in \BD'_0(U)$ to the space of test functions $C^\i_c(U)$ induces a Radon measure that is a distributional derivative of $f$.
	\end{lemma}
	
	Informally, one could phrase the statement $Df \in \BD'_0(U)$ by saying that $Df$ is concentrated on the interior of $U$, $Bf \in \BD'_{\p U}(U)$ as $Bf$ being concentrated on the boundary $\p U$, and $I f \in \BD'_\i(U)$ as $If$ being concentrated at infinity.
	
	\begin{proof}
		Consider the linear functional
		$$
		L_f \colon \DM(U) \to \R \colon u \mapsto - \int_U f \, \mathrm{d} \DIV u.
		$$
		The functional $L_f$ is well-defined because $f$ is $\bar{\BB}(U)$-measurable and of bounded fluctuation. Since $\| f \|_{BF} < + \i$, the functional is Lipschitz continuous with respect to the essentially uniform norm $\| \cdot \|_\i$ so that it may be uniquely extended to a linear and continuous functional $\mu_f$ on $\BD(U)$. By construction, the first of \eqref{eq:mufgauß} is true. The second follows from the first by taking the supremum over all $u \in \DM(U)$ such that $\| u \|_\i \le 1$. Moreover, if $\mu_f = Df + Bf + If$, then \eqref{eq:mufnorm} is immediate by \Cref{cor:DBdeco}. Finally, if $Df \in \BD'_0(U)$ and $u \in \DM(U)$ has compact support in $U$, then
		$$
		- \int_U f \, \mathrm{d} \DIV u = \langle \mu_f, u \rangle = \langle Df, u \rangle.
		$$
		Specializing to test functions $u \in C^\i_c(U; \R^m)$, the addendum follows.
	\end{proof}
	
	\begin{corollary} \label{cor:muf}
		Let $U$ have finite anexometer. Then $D \chi_U = 0$, i.e., the surface functional $\mu_{\chi_U}$ is concentrated on the boundary $\p U$ and, possibly, at infinity, but vanishes in the interior of $U$.
	\end{corollary}
	
	\begin{proof}
		By \Cref{lem:muf}, it is to be proved that
		\begin{equation} \label{eq:2bproved}
			- \int_U \, \mathrm{d} \DIV u = \langle \mu_{\chi_U}, u \rangle = 0
		\end{equation}
		for every $u \in \DMN(U)$, where $\DMN(U)$ is the closure of compactly supported elements in $\DM(U)$. By continuity, it suffices to consider $u$ having compact support. By \Cref{lem:MSinDM}, there is a sequence of regular vector fields $u_k$ satisfying
		$$
		u_k \in \DM(U) \cap C^\i(U; \R^m), \quad \DIV u_k \in L^1(U) \cap C_b(U)
		$$
		and such that $\DIV u_k$ converges to $\DIV u$ weakly* in $\MM(U)$ and $\| \DIV u_k \|(U)$ converges to $\| \DIV u \|(U)$. In addition, since $u$ has compact support in $U$, multiplying $u_k$ with a smooth cut-off function, $u_k$ can be arranged to vanish outside of a compact subset of $U$. Up to extracting a further subsequence, the positive and negative parts $\left[ \DIV u_k \right]^\pm$ converge weakly* to limits $\mu^\pm$. By weak* lower semicontinuity of the norm, $\| \left[ \DIV u_k \right]^\pm \|(U)$ converge to $\| \mu^\pm \|(U)$. But then, we have
		$$
		\DIV u_k (U) = \left[ \DIV u_k \right]^+ (U) - \left[ \DIV u_k \right]^- (U) \to \mu^+(U) - \mu^-(U) = \DIV u (U).
		$$
		Hence, it suffices to consider $u \in C^\i_c(U; \R^m)$, for which \eqref{eq:2bproved} is known to hold.
	\end{proof}
	
	It is interesting to observe that if $U$ has finite anexometer and $f \in \BF(U)$, then $B f$ and $I f$ depend on $B \chi_U$ and $I \chi_U$ in a transparent manner since studying these addends for any $f$ then essentially reduces to understanding $\mu_{\chi_U}$. This is recorded in the next corollary.
	
	\begin{corollary} \label{cor:muf2}
		Let $U$ have finite anexometer and $f \in \BF(U)$. Then
		$$
		\mu_f = Df + f \mu_{\chi_U} = Df + f B \chi_U + f I \chi_U \quad \text{ in } \BD'(U).
		$$
	\end{corollary}
	
	\begin{proof}
		If $\varphi \in C^1_c(\R^m)$ and $u \in \DM(U)$ then
		$$
		\DIV \left( \varphi u \right) = u \cdot \nabla \varphi + \varphi \DIV u
		$$
		in the sense of distributions by \Cref{lem:MSinDM}. Therefore, we have
		\begin{equation} \label{eq:scramt}
			\begin{aligned}
				\langle \DIV \left( f u\right), \varphi \rangle
				= \int_U \varphi \, \mathrm{d} \DIV \left( f u \right)
				& = - \int_U f \nabla \varphi \cdot u \, \mathrm{d}x \\
				& = - \int_U f \, \DIV \left( \varphi u \right) + \int_U f \varphi \, \mathrm{d} \DIV u \\
				& = \langle D f, \varphi u \rangle + \int_U f \varphi \, \mathrm{d} \DIV u.
			\end{aligned}
		\end{equation}
		In particular, the divergence $\DIV \left( f u \right)$ belongs to $\MM(\R^m)$ so that $fu \in \DM(U)$. The compact support of $\varphi u$ was used in the last step. We rearrange the terms in \eqref{eq:scramt} to find
		\begin{equation} \label{eq:unscramt}
			- \int_U f \varphi \, \mathrm{d} \DIV u = \langle D f, \varphi u \rangle - \int_U \varphi \, \mathrm{d} \DIV \left( f u \right).
		\end{equation}
		Choosing for $\varphi$ in \eqref{eq:unscramt} a sequence $\varphi_n \in C^1_c(U; [0, 1] )$ that converges to $\chi_U$ uniformly on compact subsets of $U$, using that $Df \in \BD'_0(U)$ is concentrated on the interior of $U$ in the sense specified by the possibility of approximating the projector onto $\BD'_0(U)$ as described in the addendum of \Cref{cor:DBdeco}, we obtain
		$$
		\langle D f, \varphi_n u \rangle \to \langle D f, u \rangle.
		$$
		Thus, sending $n \to + \i$, we find
		\begin{gather*}
			\langle \mu_f, u \rangle = - \int_U f \, \mathrm{d} \DIV u = \langle D f, u \rangle + \langle \mu_{\chi_U}, f u \rangle \\
			= \langle D f, u \rangle + \langle B \chi_U, f u \rangle + \langle I \chi_U, f u \rangle \quad \forall u \in \DM(U),
		\end{gather*}
		where \Cref{cor:muf} entered in the last step.
	\end{proof}
	
	Together with \Cref{thm:main}, the next corollary generalizes \cite[Thm. 4.2]{CLT}.
	
	\begin{corollary} \label{cor:extdom}
		Let $f \in \BF(U)$ and $u \in \DM(U)$. Then $f u \in \DM(\R^m)$ if $f u$ is extended to $\R^m \setminus U$ by zero. In particular, if $\chi_U \in \BF(U)$, then $U$ is a null extension domain for $\DM(U)$, i.e., there holds $\chi_U u \in \DM(\R^m)$ whenever $u \in \DM(U)$. 
	\end{corollary}
	
	\begin{proof}
		If $\varphi \in C^\i_c(\R^m)$, then
		$$
		\DIV \left( \varphi u \right) = u \cdot \nabla \varphi + \varphi \DIV u
		$$
		in the sense of distributions by \Cref{lem:MSinDM}. Therefore $\varphi u \in \DM(U)$ so that
		$$
		- \langle \mu_f, \varphi u \rangle = \int_U f \, \mathrm{d} \DIV \left( \varphi u \right) = \int_U f u \cdot \nabla \varphi \, \mathrm{d} x + \int_U f \varphi \, \mathrm{d} \DIV u.
		$$
		Rearranging terms and invoking \Cref{lem:muf} gives
		$$
		\int_U f u \cdot \nabla \varphi \, \mathrm{d} x
		\le \| \mu_f \| \| u \|_\i \| \varphi \|_\i + \| f \|_\i \| \DIV u \|(U) \| \varphi \|_\i
		$$
		so that $\DIV \left( f u \right) \in \MM(\R^m)$. Since $f u \in L^\i(\R^m)$ by $\| f u \|_\i \le \| f \|_\i \| u \|_\i$, the proof is complete.
	\end{proof}
	
	\emph{Example:}	Let $\WW_0(U)$ be the homogeneous Sobolev space arising as the closure of test functions $C^\i_c(U)$ in the Poincaré norm $\| f \| = \| \nabla f \|_1$. If $f \in C^\i_c(U)$, then it is immediate to check that $\nabla f = D f = \mu_f$, hence $\| f \|_{BF(U) } = \| \nabla f \|_1$ and by approximation the same is true for the precise representative of any $f \in \WW_0(U) \cap L^\i(U)$.
	\\
	This observation applies in particular if $f \in C_0(U)$ with $\nabla f \in L^1(U)$. To see this, it suffices to prove that the positive part $f^+$ and negative part $f^-$ belong to $\WW_0(U)$. Let $\e > 0$ and $f^+_\e = (f - \e)^+$. The function $f^+_\e$ is continuous with compact support in $U$ since $f \in C_0(U)$. Moreover, $\nabla f^+_\e \in L^1(U)$ by \cite[Thm. 4.4(iii)]{EG}. A standard convolution smoothing argument shows that $f^+_\e \in \WW_0(U)$. Invoking again \cite[Thm. 4.4(iii)]{EG} gives $f^+_\e \to f^+$ in $\WW_0(U)$ as $\e \to 0$. The argument is analogous for $f^-$ so that in total $f \in \WW_0(U) \cap L^\i(U)$ whenever $f \in C_0(U)$ with $\nabla f \in L^1(U)$.
	\\
	
	By the Hahn-Banach extension theorem, \Cref{lem:muf} implies the existence of a finitely additive set function that vanishes on Lebesgue null sets of $U$, i.e., an element $\bar{\mu}_f \in \ba(U; \R^m) \cong L'_\i(U; \R^m)$, for which \eqref{eq:mufgauß} continues to hold. Conversely, the existence of such a set function suffices for $f$ to have bounded fluctuation by \eqref{eq:bfseminorm}. In this sense, my formulation of the divergence theorem implies the one first used in \cite[Thm. 3.18]{SS}. Cf. also \cite[Cor. 4.29]{SS2, SS3}. Conversely, the formulation in these sources is clearly sufficient for $U$ to have finite anexometer. While $\BD'(U)$ is isometrically isomorphic to the quotient space $\ba(U; \R^m) / \BD(U)^\circ$ if $\BD(U)^\circ$ denotes the polar space of $\BD(U) \subset L^\i(U; \R^m)$, it were also interesting to know if it can be identified directly with a vector space of (more regular) finitely additive $\R^m$-valued set functions, which would enable integration with respect to the surface functional without the need to choose a representative for it in $\ba(U; \R^m)$.
	\\
	
	I conclude this section by characterizing the kernel of the bounded fluctuation seminorm. Besides its use in the present investigation, this result provides an intrinsically interesting characterization of the extent to which two functions of bounded fluctuation may differ while still defining essentially the same element of $\BF(U)$. While the implication $\implies$ of \Cref{thm:semiker} follows easily from the absolute $\HH^{m - 1}$-continuity of the Radon measure $\DIV u$ for $u \in \DM(U)$, the converse implication is substantially complicated by the need to pass from the precise representative to the function itself. Indeed, if we contented ourselves with the precise representative, then \Cref{thm:semiker} could be obtained as a corollary to the general theory of bounded variation functions.
	
	\begin{theorem} \label{thm:semiker}
		Let $f \colon U \to \R$ be a bounded $\bar{\BB}(U)$-measurable function. Then
		$$
		f \text{ vanishes } \HH^{m - 1} \text{-a.e.} \iff \| f \|_{BF(U) } = 0.
		$$
	\end{theorem}
		
	\begin{proof}
		$\implies$: if $f$ vanishes $\HH^{m - 1}$-a.e., then
		$$
		\int_U f \, \mathrm{d} \DIV u = 0 \quad \forall u \in \DM(U)
		$$
		since $\DIV u \ll \HH^{m - 1}$ by \cite[Cor. 2.3.1]{Co}. Thus, $\| f \|_{BF(U) } = 0$ by \Cref{def:bndfluc}.
		\\
		$\impliedby$: we may assume $f$ to be $\BB(U)$-measurable by modifying it on an $\HH^{m - 1}$-null set. Moreover, $\| f \|_{BF(U) } = 0$ implies that $f = 0$ in the sense of distributions by \Cref{lem:div rich}, hence $f$ vanishes $\mathcal{L}^m$-a.e. In particular, $f \in \mathcal{W}^{1, 1}_0(U) \cap \mathcal{L}^\i(U)$. Consider the locally convex space $X = W^{1, 1}_0(U) \cap L^\i(U)$ topologized in the coarsest way that renders the projections $X \to W^{1, 1}_0(U)$ and $X \to L^\i(U)$ continuous, where $W^{1, 1}_0(U)$ carries its norm topology and $L^\i(U)$ its bounded weak* topology. This is the topology induced by the seminorms $s(x', K)$, where $K$ runs over the absolutely convex, norm compact subsets of the predual space $L^1(U)$. Crucially, the bounded weak* topology is thus locally convex and has the same continuous, linear functionals as the weak* topology by the Mackey-Arens theorem since it is finder than the weak* topology but coarser than the Mackey topology, which is defined by letting $K$ run over the absolutely convex, weakly compact subsets. Crucially, the bounded weak* topology has an equivalent description as the final topology that renders the inclusion of any norm bounded set into $L^\i(U)$ continuous when $L^\i(U)$ carries the weak* topology. Thus, a linear functional is bounded weak* continuous iff its restriction to any norm bounded set is weak* continuous. In particular, we may restrict to checking continuity along bounded nets.	Arguing similar to the case of intersected Banach spaces, cf., e.g., \cite[Kap. 1, Satz 5.13]{GGZ}, one may show that the continuous dual of $X$ is given by $W^{- 1, \i}(U) + L^1(U)$, i.e., if $x' \in X'$, then there exist $x'_1 \in W^{- 1, \i}(U) \cong W^{1, 1}_0(U)'$ and $x'_2 \in L^1(U)$ such that $x' = x'_1 + x'_2$.
		\\
		Let $K \subset U$ be compact with $\HH^{m - 1}(K) < + \i$. We denote the precise representative of a function $u \in W^{1, 1}_0(U)$ by $u^*$. Note that $u^*$ will be $\HH^{m - 1}$-essentially bounded if $u \in X$ due to \cite[Thm. 5.20(ii)]{EG}.	We claim that
		$$
		\ell \colon X \to \R \colon u \mapsto \int_K u^* \, \mathrm{d} \HH^{m - 1}
		$$
		defines an element of $X'$. Obviously, $\ell$ is linear. To check its continuity, we may reduce to nets which are norm bounded in $L^\i(U)$ since a linear functional is bounded weak* continuous iff its restriction to the unit ball is weak* continuous. Hence, we may even reduce to sequences since the weak* topology on $L^\i(U)$ is metrizable on bounded sets by merit of $L^1(U)$ being separable. Moreover, since $K \subset \subset U$, there is no loss in assuming $U = \R^m$. Let $u_n \in X$ be a (bounded) sequence with $\lim_n u_n = 0$. Arguing by convolution smoothing with the standard mollifier, we find sequences $u_{n, k} \in X \cap C^\i(U)$ with $\lim_k u_{n, k} = u_n$ in $X$ and $\HH^{m - 1}$-a.e. by \cite[Thm. 5.20(ii)]{EG}. In particular, $\lim_k \ell(u_{n, k} - u_n) = 0$ so that upon replacing $u_n$ with $u_{n, k(n) }$ for a sufficiently large $k = k(n)$, we may assume $u_n \in X \cap C^\i(U)$ for the purpose of proving $\lim_n \ell(u_n) = 0$. Thus, invoking \cite[Thm. 4.22, Cor. 4.25]{Ki} and extracting a subsequence (not relabeled), we have $\lim_n u_n(x) = 0$ for $\HH^{m - 1}$-a.e. $x \in U$. Combining this with the dominated convergence theorem and the boundedness of $u_n$ in $L^\i(U)$, we find $\lim_n \ell(u_n) = 0$ as had to be shown. Consequently, $\ell$ is continuous so that there exist $\ell_1 \in L^\i(U; R^m)$ and $\ell_2 \in L^1(U)$ with
		$$
		\int_K u^* \, \mathrm{d} \HH^{m - 1} = - \int_U \nabla u \cdot \ell_1 \, \mathrm{d} x + \int_U u \ell_2 \, \mathrm{d} x \quad \forall u \in X.
		$$
		In particular, by picking $u \in C^1_c(U)$ and rearranging terms, we see that the distributional derivative $\DIV \ell_1$ is a Radon measure. Moreover, since every element of $C^1_c(U)$ agrees with its precise Sobolev representative, we see that $\HH^{m - 1} \resmes K = \DIV \ell_1 + \ell_2$ as Radon measures. Consequently, setting $f_{a, b} \coloneqq \min\{ a, \max\{ - b, f \} \}$, we have
		$$
		\int_K f_{a, b} \, \mathrm{d} \HH^{m - 1} = - \int_U \nabla f_{a, b} \cdot \ell_1 \, \mathrm{d} x + \int_U f_{a, b} \ell_2 \, \mathrm{d} x \quad \forall f \in \mathcal{W}^{1, 1}_0(U),
		$$
		with $f_{a, b}$ instead of its precise representative on the left side. We conclude
		\begin{equation} \label{eq:int f itself}
			\lim_{a, b \to + \i} \int_K f_{a, b} \, \mathrm{d} \HH^{m - 1} = - \int_U \nabla f \cdot \ell_1 \, \mathrm{d} x + \int_U f \ell_2 \, \mathrm{d} x \quad \forall f \in \mathcal{W}^{1, 1}_0(U) \cap \mathcal{L}^\i(U).
		\end{equation}
		If $\| f \|_{BF(U) } = 0$, then
		\begin{equation} \label{eq:vani}
			- \int_U \nabla f \cdot \ell_1 \, \mathrm{d} x = 0.
		\end{equation}
		Moreover, since $f$ vanishes $\mathcal{L}^m$-a.e. by \Cref{lem:div rich}, we have
		\begin{equation} \label{eq:vani2}
			\int_U f \ell_2 \, \mathrm{d} x = 0.
		\end{equation}
		Putting \eqref{eq:int f itself}, \eqref{eq:vani}, and \eqref{eq:vani2} together implies
		\begin{equation} \label{eq:vani3}
			\lim_{a, b \to + \i} \int_K f_{a, b} \, \mathrm{d} \HH^{m - 1} = 0 \quad \forall K = \cl K \subset \subset U \colon \HH^{m - 1}(K) < + \i.
		\end{equation}
		As every Suslin set $S \subset U$ is inner regular for $\HH^{m - 1}$ by \cite[Cor. 2.10.48]{Fe3} and $f$ is Borel measurable, we conclude that $f$ vanishes $\HH^{m - 1}$-a.e. For, if $\HH^{m - 1} \left( \left\{ f > r  \right\} \right) > 0$ for some $r \in R$, then, since $\left\{ f > r \right\}$ is a Borel set, hence a Suslin set by \cite[§2.2.10, p. 66]{Fe3}, we could pick a compact set $K \subset \left\{ f > r \right\}$ with $\HH^{m - 1}(K) > 0$, contradicting \eqref{eq:vani3}.
	\end{proof}
	
	\subsection{Approximation by smooth functions}
	
	Any $f \in \BF(U)$ is $\bar{\BB}(U)$-measurable so that it agrees with an $\BB(U)$-measurable function outside of a $\HH^{m - 1}$-null set by \cite[Lem. A.3]{Ru}. Clearly, there is such a modification whose range is no larger than that of $f$.
	
	\begin{theorem} \label{thm:innerappro}
		For any real numbers $a \le 0 \le b$, for every $f \in \BF(U; [a, b] )$ with $\| f \|_{BF} > 0$ and every $\BB(U)$-measurable function $\mathfrak{f} \colon U \to \left[ a, b \right]$ satisfying $\mathfrak{f}(x) = f(x)$ for $\HH^{m - 1}$-a.e. $x \in U$, there is a net of functions
		$$
		f_\alpha \in C_0(U; [a, b] ); \quad \nabla f_\alpha \in L^1(U); \quad \| \nabla f_\alpha \|_1 \le \| f \|_{BF} \quad \forall \alpha \in I
		$$
		such that
		$$
		f_\alpha \weakast \mathfrak{f} \text{ in } \MM'(U); \quad f_\alpha \weakast f \text{ in } \MM'_{\DIV}(U); \quad \nabla f_\alpha \weakast \mu_f \text{ in } \BD'(U).
		$$
	\end{theorem}
	
	\emph{Remark:} \Cref{thm:innerappro} will fail if
	\begin{equation} \label{eq:f BF null}
		\| f \|_{BF} = 0
	\end{equation}
	because no discontinuous modification of $f$ can be approximated weakly* in $\MM'(U)$ by the constant null net. Trivially, the null net still provides a smooth approximation of the null function, which is an $\HH^{m - 1}$-a.e. modification of $f$ if \eqref{eq:f BF null} by \Cref{thm:semiker}.
	
	\begin{proof}
		We first treat the case when $0 \in (a, b)$. Let $X = L^1(U; \R^m ) \times C_0(U)$ and consider the following convex subsets of $X$:
		\begin{gather*}
			K_1 = \left\{ \| x_1 \| \le \| f \|_{BF(U) } \right\}; \quad
			K_2 = \left\{ x_2 \in C_0(U;[a, b]) \right\}; \\
			K_3 = \left\{ \langle x_1, u \rangle = - \langle x_2, \DIV u \rangle \quad \forall u \in \DM(U) \right\}.
		\end{gather*}
		The origin is an interior point of $K_1$ and $K_2$ in the strong topology of $X$. Combining this with $0 \in K_1 \cap K_2 \cap K_3$, the duality theorem \cite[§3.4, Thm. 1]{IT} together with the observation that $I_A^* = s_A$ for every subset $A \subset X$ implies
		\begin{equation} \label{eq:inficonv1}
			s \left( K_1 \cap K_2 \cap K_3, \cdot \right) = s(K_1, \cdot) \oplus s(K_2, \cdot) \oplus s(K_3, \cdot).
		\end{equation}
		Here, the symbol $\oplus$ denotes infimal convolution of functions. Now, applying \cite[§3.4, Thm. 1]{IT} to compute the convex conjugate of the function in \eqref{eq:inficonv1} for the dual pair $\left( X', X'' \right)$ yields
		\begin{equation} \label{eq:clintersect2}
			\begin{gathered}
				I \left( K_1 \cap K_2 \cap K_3, \cdot \right)^{**} = I \left( K_1, \cdot \right)^{**} + I \left( K_2, \cdot \right)^{**} + I \left( K_3, \cdot \right)^{**}  \\
				\iff
				\mathrm{cl}^* \left( K_1 \cap K_2 \cap K_3 \right) = \mathrm{cl}^* K_1 \cap \mathrm{cl}^* K_2 \cap \mathrm{cl}^* K_3,
			\end{gathered}
		\end{equation}
		where the sets $K_1, K_2, K_3$ are identified with their images under the canonical embedding of $X$ into $X''$ and the closure is taken in the weak* topology of $X''$. The Goldstine theorem \cite[Satz VIII.3.17]{We} implies
		\begin{equation} \label{eq:clos1}
			\mathrm{cl}^* K_1 = \left\{ x^{**} \in X'' \colon \| x^{**}_1 \| \le \| f \|_{BF(U) } \right\}.
		\end{equation}
		Denoting by $\delta_u$ the Dirac measure at a point $u \in U$, we check that
		\begin{equation} \label{eq:clos2}
			\mathrm{cl}^* K_2 = \left\{ x^{**} \in X'' \colon \langle \delta_u, x^{**}_2 \rangle_{X', X''} \in [a, b] \quad \forall u \in U \right\}.
		\end{equation}
		Indeed, the polar set of the right set in \eqref{eq:clos2} with respect to the pair $(X', X'')$ is a weak* closed convex set containing the origin. Moreover, $K_2$ is the polar of this set with respect to the pair $(X, X')$. Consequently, $K_2$ is weak* dense in the right set by \Cref{thm:predualdens}. Claim:
		\begin{equation} \label{eq:clos3}
			\mathrm{cl}^* K_3 = \left\{ (x^{**}_1, x^{**}_2) \in X'' \colon \langle u, x^{**}_1 \rangle = - \langle \DIV u, x^{**}_2 \rangle \quad \forall u \in \DM(U) \right\}.
		\end{equation}
		Indeed, under the isometric identification
		$$
		\DM(U) \to X' = L^\i(U) \times \MM(U) \colon u \mapsto (u, \DIV u),
		$$
		the set $K_3$ becomes the polar of $\DM(U)$ with respect to the dual pair $\left( X, X' \right)$. Moreover, $\DM(U)$ is weak* closed in $X'$ because, if $u_\alpha$ converges to $u$ weakly* in the Lebesgue space $L^\i(U; \R^m )$ and $\DIV u_\alpha$ converges to $\mu$ weakly* in $\MM(U)$, then, as $(\nabla \varphi, \varphi) \in X$ for any test function $\varphi \in C^\i_c(U)$, there holds
		$$
		- \int_U \varphi \, \mathrm{d}\mu = - \lim_\alpha \int_U \varphi \, \mathrm{d} \DIV u_\alpha = \lim_\alpha \int_U \nabla \varphi \cdot u_\alpha \, \mathrm{d}x = \int_U \nabla \varphi \cdot u \, \mathrm{d}x,
		$$
		so that $u \in L^\i(U; \R^m)$ and $\DIV u = \mu \in \MM(U)$ in $\DD'(U)$ and hence $u \in \DM(U)$. Therefore, \eqref{eq:clos3} follows by \Cref{thm:predualdens}.
		\\
		Let $\mu_{\mathfrak{f} }$ denote a norm-preserving Hahn-Banach extension of $\mu_f$ to $L^\i(U)$. Clearly, $\left( \mu_{\mathfrak{f} }, \mathfrak{f} \right) \in \mathrm{cl}^* K_2$ by \eqref{eq:clos2}. Moreover, $\left( \mu_{\mathfrak{f} }, \mathfrak{f} \right) \in \mathrm{cl}^* K_3$ by \eqref{eq:clos3} and $f \in \BF(U)$. In total, \eqref{eq:clintersect2} implies the existence of a net $\left( \nabla f_\alpha, f_\alpha \right) \in K_1 \cap K_2 \cap K_3$ such that $\left( \nabla f_\alpha, f_\alpha \right)$ converges weakly* to $\left( \mu_{\mathfrak{f} }, \mathfrak{f} \right)$ in $X''$. By definition of $K_1$, $K_2$, $K_3$ this net has the desired properties.
		\\
		It remains to remove that restriction $0 \in (a, b)$. Let $\NN(\mathfrak{f})$ be a neighborhood base of $\mathfrak{f}$ in the weak* topology of $\MM'(U)$. For every $\e > 0$ and $V \in \NN(\mathfrak{f})$, there is $f_{V, \e} \in C_0(U; \left[ a - \e, b + \e \right] )$ with $\| \nabla f_{V, \e} \|_{L^1} \le \| f \|_{BF}$ and $f_{V, \e} \in V$. Let $V_n \in \NN(\mathfrak{f})$ be any sequence with $V_{n + 1} \subset V_n$ for all $n \in \N$. Regard $\NN(\mathfrak{f})$ as a directed set via downward inclusion. Define a net $f_V$ that converges weakly* to $\mathfrak{f}$ by choosing $f_V \in V$ with $\| \nabla f_V \|_{L^1} \le \| f \|_{BF}$ and $V \subset V_n \implies f_V \in C_0(U; \left[ a - \tfrac{1}{n}, b + \tfrac{1}{n} \right] )$. Finally, replace $f_V$ by the truncated net $\mathfrak{f}_V = \max\{ a, \min\{ b, f_V \} \}$, which satisfies $\| \nabla \mathfrak{f}_V \|_{L^1} \le \| \nabla f_V \|_{L^1} \le \| f \|_{BF}$ by \cite[Thm. 4.4(iii)]{EG} and $V \subset V_n \implies \| f_V - \mathfrak{f}_V \| \le \frac{1}{n} \to 0$ so that $\mathfrak{f}_V$ also converges weakly* to $\mathfrak{f}$. Moreover, $\nabla \mathfrak{f}_V \in \BD'(U)$ is relatively compact and clearly, any cluster point of the net equals $\mu_f$ so that $\nabla \mathfrak{f}_V$ converges to $\mu_f$ weakly* in $\BD'(U)$.
	\end{proof}
	
	\begin{corollary} \label{cor:testdens}
		The net in \Cref{thm:innerappro} may be chosen in $C^\i_c(U; [a, b] )$.
	\end{corollary}
	
	\begin{proof}
		By \Cref{thm:innerappro}, it suffices to consider $f \in C_0(U; [a, b] )$ with $\nabla f \in L^1(U)$. Such an $f$ can be strongly approximated by test functions as described in the example below \Cref{lem:muf}. Moreover, these satisfy $\sup_\alpha \| \nabla f_\alpha \|_1 \le \| f \|_{BF(U) }$ by the Young convolution inequality and belong to $C^\i_c(U; [a, b] )$ if $f \in C_0(U; [a, b] )$.
	\end{proof}
	
	\begin{corollary} \label{cor:GNSinBF}
		Let $1^* = \tfrac{m}{m - 1} \in (1, +\i]$ for $m \ge 1$.
		\begin{equation} \label{eq:GNSinBF}
			\| f \|_{L^{1^*}(U) } \le \omega_m \| f \|_{BF(U) } \quad \forall f \in \BF(U),
		\end{equation}
		where $\omega_m$ is the optimal constant in the Sobolev inequality \cite[Thm. 4.8]{EG}.
	\end{corollary}
	
	\begin{proof}
		By \cite[Thm. 4.8]{EG} in connection with the example below \Cref{lem:muf}, the inequality \eqref{eq:GNSinBF} is true for all $f \in C^\i_c(U)$. By \Cref{cor:testdens}, for any $f \in \BF(U)$, there is a net $f_\alpha \in C^\i_c(U)$ such that $\lim \| f_\alpha \|_{BF(U) } = \| f \|_{BF(U) }$ and $f_\alpha$ converges to $f$ weakly* in $\MM'_{\DIV}(U)$. By \eqref{eq:GNSinBF} being true for $f_\alpha$, the net $f_\alpha$ is bounded in $L^{1^*}(U)$ and $\lim \langle f_\alpha, g \rangle = \langle f, g \rangle$ for $g \in \MM'_{\DIV}(U)$, which is a dense subset of the dual of $L^{1^*}(U)$ so that $f_\alpha$ converges to $f$ weakly in $L^{1^*}(U)$, cf. \Cref{lem:div rich}. Consequently, \eqref{eq:GNSinBF} follows by weak lower semicontinuity of the norm.
	\end{proof}
	
	\begin{corollary} \label{cor:finmass}
		Let $m \ge 2$ and $\left| U \right| = + \i$. Then $U$ has infinite anexometer.
	\end{corollary}
	
	\begin{proof}
		Let $U$ have finite anexometer. Then there is a net $f_\alpha$ approximating $\chi_U$ as in \Cref{thm:innerappro}. By \Cref{cor:GNSinBF},
		\begin{align*}
			\left| U \right|^{1 / 1^*} = \| \chi_U \|_{L^{1^*}(U) } \le \liminf \| f_\alpha \|_{L^{1^*}(U) }
			& \le \omega_m \lim \| \nabla f_\alpha \|_1 \\
			& = \omega_m \| \chi_U \|_{BF(U) } < + \i \qedhere.
		\end{align*}
	\end{proof}
	
	\Cref{cor:finmass} may fail for $m = 1$. For example, $\R$ has finite anexometer as
	$$
	\| \chi_\R \|_{BF(\R) } = \sup \left\{ u(+ \i) - u(- \i) \colon u \in \DM(\R), \, \| u \|_\i \le 1 \right\} = 2.
	$$
	Here, we denote the limit of $u(t)$ as $t \to \pm \i$ by $u(\pm \i)$. The limit exists because the gradient of $u \in \DM(\R)$ is a finite Radon measure.
	
	\begin{corollary} \label{cor:innerappro2}
		If $f \in \BF(U; [a, b] )$ is continuous on $U$, then there is a \emph{sequence} $f_n \in C^\i_c(U; [a, b] )$ satisfying $\sup_n \| \nabla f_n \|_1 \le \| f \|_{BF(U) }$ such that $f_n$ converges to $f$ uniformly on every compact set $K \subset U$ and $\nabla f_n$ converges to $\mu_f$ weakly* in $\BD'(U)$.
	\end{corollary}
	
	\begin{proof}
		Let $g_\alpha$ be an approximating net for $f$ as in \Cref{thm:innerappro}. One may choose $g_\alpha \in C^\i_c(U; [a, b] )$ by \Cref{cor:testdens}. Then the restriction $\left. g_\alpha \right|_K$ converges to $\left. f \right|_K$ weakly in $C(K)$ because $g_\alpha$ converges to $f$ weakly* in $\MM'(U)$. By \cite[Ch. 4, §4, IV.48, Prop. 14]{Bo} and the subsequent remark, this implies that $g_\alpha$ converges to $f$ weakly in the Frechet space of continuous functions $C(U)$ equipped with the topology of uniform convergence on compact sets. Therefore, $f$ belongs to the closed convex hull of $\left\{ g_\alpha \right\}$ in this space by the Hahn-Banach separation theorem for closed convex sets. Hence, as $C(U)$ is a metric space, there is a sequence $f_n$ belonging to the convex hull of the net $\left\{ g_\alpha \right\}$ such that $f_n$ converges to $f$ uniformly on compact subsets of $U$. As $f_n$ belongs to the convex hull of $\left\{ g_\alpha \right\}$, it satisfies $f_n \in C_0(U; [a, b] )$ and $\sup_n \| \nabla f_n \|_{L^1} \le \| f \|_{BF(U) }$. By the Banach-Alaoglu theorem, the sequence $\nabla f_n$ is relatively weakly* compact in $\BD'(U)$. It is immediate that every weak* limit of a convergent subnet agrees with the unique functional $\mu_f \in \BD'(U)$ satisfying \eqref{eq:mufgauß}. Therefore, $\nabla f_n$ converges to $\mu_f$ weakly* in $\BD'(U)$ as claimed.
	\end{proof}
	
	\subsection{Another seminorm}
	
	Alas, by contrast to the BV seminorm, it is not clear if the bounded fluctuation seminorm enjoys good lower semicontinuity properties along a sequence that converges a.e. with respect to Lebesgue measure. To overcome this obstacle, we introduce another seminorm for which this property can be shown and which agrees with the BF seminorm on sufficiently regular functions.	Let
	$$
	L \coloneqq \left\{ u \in \DM(U) \cap C^\i(U) \colon \DIV u \in L^1(U) \cap C_b(U) \right\}.
	$$
	For $f \in L^1(U)$, consider the seminorm
	$$
	\vert f \vert_{BF(U) } \coloneqq \sup \left\{ \int_U f \, \mathrm{d} \DIV u \colon u \in L, \, \| u \|_\i \le 1 \right\},
	$$
	which is dominated by $\| \cdot \|_{BF}$ wherever both are defined as $L \subset \DM(U)$.
	
	\begin{proposition} \label{prop:lscseminorm}
		Let $f_\alpha \in L^1(U)$ be a net converging to $f \in L^1(U)$ in $L^1(V)$ for every bounded subset $V \subset U$. Then
		\begin{equation} \label{eq:lscseminorm}
			\vert f \vert_{BF(U) } \le \liminf \vert f_\alpha \vert_{BF(U) }.
		\end{equation}
	\end{proposition}
	
	\begin{proof}
		Let $V_n \coloneqq B_n(0) \cap U$, $\bar{\eta}$ be the standard convolution kernel and $\eta(x) \coloneqq 2^{-m} \bar{\eta} \left( \tfrac{x}{2} \right)$ so that $\supp(\eta) = \bar{B}_{\tfrac{1}{2} }(0)$. The sequence of functions $\varphi_n \coloneqq \eta * \chi_{V_n}$ is bounded in $W^{1, \i}(U; [0, 1] )$ and satisfies $\varphi_n(x) = 1$ if $x \in V_n$ and $\varphi_n(x) = 0$ if $x \in \R^m \setminus V_{n + 1}$. Define the multiplication operators $M_n \colon \DM(U) \to \DM(U) \colon u \to \varphi_n u$. Then, if $B_L = \left\{ u \in L \colon \| u \|_\i \le 1 \right\}$, there holds
		\begin{equation} \label{eq:isotony}
			M_n(B) \subset M_{n + 1}(B) \subset B
		\end{equation}
		because $\varphi_n \varphi_{n + 1} = \varphi_n$. Moreover, for every $u \in B_L$, there holds $\DIV \left( \varphi_n u \right) = \nabla \varphi_n \cdot u + \varphi_n \DIV u$. Consequently, whenever $u \in B_L$ and $f \in L^1(U)$, since $\varphi_n$ is uniformly bounded and $\nabla \varphi_n$ converges to the constant null function uniformly on every bounded set, the dominated convergence theorem implies
		\begin{equation} \label{eq:pointwisecvg}
			\lim_n \int_U f \, \mathrm{d} \DIV \left( \varphi_n u \right) = \int_U f \, \mathrm{d} \DIV u.
		\end{equation}
		Combining \eqref{eq:isotony} and \eqref{eq:pointwisecvg} implies
		$$
		\vert f \vert_{BF(U) } = \sup \left\{ \int_U f \, \mathrm{d} \DIV u \colon u \in B \right\} = \sup_n \sup \left\{ \int_U f \, \mathrm{d} \DIV u \colon u \in M_n(B) \right\}.
		$$
		Since every function
		$$
		f \mapsto \sup \left\{ \int_U f \, \mathrm{d} \DIV u \colon u \in M_n(B) \right\}, \quad n \in \N,
		$$
		is lower semicontinuous along the net $f_\alpha$ by being a supremum of continuous functions, the same is true for their supremum, which is \eqref{eq:lscseminorm}.
	\end{proof}
	
	\begin{lemma} \label{lem:seminormid}
		Let $f \colon U \to \R$ be a bounded, $\bar{\BB}(U)$-measurable function. If
		\begin{equation} \label{eq:disco null}
			\HH^{m - 1} \left( \left\{ x \in U \colon f \text{ is discontinuous in } x \right\} \right) = 0,
		\end{equation}
		then
		\begin{equation} \label{eq:normid}
			\| f \|_{BF(U) } = \vert f \vert_{BF(U) }.
		\end{equation}
	\end{lemma}
	
	\begin{proof}
		The statement remains true if $f$ is modified on an $\HH^{m - 1}$-null set by \Cref{thm:semiker}. Thus, we may assume $f$ to be Borel measurable. It suffices to prove that $\vert f \vert_{BF(U) } \ge \| f \|_{BF(U) }$. Given $\e > 0$, let $u \in\DM(U)$ with $\| u \|_\i \leq 1$ such that
		\begin{equation} \label{eq:normineq}
			\int_U f \, \mathrm{d} \DIV u > \min\{ \e^{-1}, \| f \|_{BF(U)} \} - \e.
		\end{equation}
		By \Cref{lem:MSinDM}, there is a sequence of functions $u_k \in \DM(U) \cap C^\i(U; \R^m)$ with $\DIV u_k \in L^1(U) \cap C_b(U)$, $\| u_k \|_\i \leq 1$ such that $\DIV u_k$ converges to $\DIV u$ weakly* in $\MM(U)$ and $\| \DIV u_k \|(U)$ converges to $\| \DIV u \|(U)$. Since $C_0(U)$ is separable, we may extract from every bounded sequence in $\MM(U)$ a weak* convergent subsequence by the sequential Banach-Alaoglu theorem. Therefore, upon passing to a subsequence (not relabeled), the positive part $[\DIV u_k]^+$ and the negative part $[\DIV u_k]^-$ converge weakly*. By the convergence of total mass and weak* lower semicontinuity of the total variation norm, $\lim_k \| [\DIV u_k]^+ \|(U) = \| [\DIV u]^+ \|(U)$, so $[\DIV u_k]^+$ even converges in the weak topology that $\MM(U)$ carries in duality with $C_b(U)$. The same is true for $[\DIV u_k]^-$. Note that $\lim_k [\DIV u_k]^+ - \lim_k [\DIV u_k]^- = \DIV u$ by the weak* convergence of $\DIV u_k$. The portmanteau theorem \cite[Thm. 13.16(iii)]{Kl} and $\DIV u \ll \HH^{m - 1} \resmes U$ together with \eqref{eq:disco null} imply
		$$
		\int_U f \, \mathrm{d} \DIV u = \lim_k \int_U f \, \mathrm{d} \DIV u_k \le \vert f \vert_{BF(U) },
		$$
		which in combination with \eqref{eq:normineq} implies $\vert f \vert_{BF(U) } > \min\{ \e^{-1}, \| f \|_{BF(U)} \} - \e$.
	\end{proof}
	
	\section{Equivalent conditions for a finite anexometer} \label{sec:divthm}
	
	\begin{proposition} \label{prop:approx2}
		If there is a net $f_\alpha \in C_0(U; [0, 1] )$ with $\sup_\alpha \| \nabla f_\alpha \|_1 < + \i$ such that $f_\alpha$ converges to $\chi_U$ weakly* in $\MM'_{\DIV}(U)$, e.g., a net provided by \Cref{thm:innerappro} if $U$ has finite anexometer, then there is a net of smooth sets $U_\alpha \subset \subset U$ with $\sup_\alpha \| D \chi_{U_\alpha} \|(U) < + \i$ such that $\chi_{U_\alpha}$ converges to $\chi_{U}$ in $L^1_{\text{loc} }(U)$.
	\end{proposition}
	
	\begin{proof}
		\Cref{lem:div rich}, the weak* convergence of $f_\alpha$ to $\chi_U$ in $\MM'_{\DIV}(U)$, and the Rellich compact embedding theorem together imply
		\begin{equation} \label{eq:falphcvg}
			f_\alpha \to \chi_U \text{ in } L^1_{\text{loc} }(U).
		\end{equation}
		The co-area formula \cite[Thm. 3.40]{AFP} implies
		$$
		\sup_\alpha \| \nabla f_\alpha \|_{L^1} = \sup_\alpha \int_0^1 P \left( \left\{ x \in U \colon f_\alpha(x) > r \right\}, U \right) \, \mathrm{d}r \le C
		$$
		so that
		\begin{equation} \label{eq:perimeterctrl}
			\forall \theta \in (0, 1) \, \forall \alpha \, \exists r_\alpha \in (0, \theta) \colon P \left( \left\{ x \in U \colon f_\alpha(x) > r_\alpha \right\}, U \right) \le \theta^{-1} C.
		\end{equation}
		Let $V_\alpha \coloneqq \left\{ x \in U \colon f_\alpha(x) > r_\alpha \right\}$. Setting $g_\alpha \coloneqq \chi_{V_\alpha}$, there holds
		\begin{equation} \label{eq:measurectrl}
			\begin{gathered}
				\left\{ 1 - g_\alpha(x) > \e \right\}
				= \left\{ 1 - g_\alpha(x) > 0 \right\}
				 = \left\{ f_\alpha(x) \le r_\alpha \right\} \\
				\subset \left\{ f_\alpha(x) \le \theta \right\}
				= \left\{ 1 - f_\alpha(x) \ge 1 - \theta \right\} \quad \forall \e \in (0, 1).
			\end{gathered}
		\end{equation}
		By \eqref{eq:falphcvg} and the Markov inequality, the local Lebesgue measure in $U$ of the last set goes to zero as $\alpha$ grows indefinitely, so that
		\begin{equation} \label{eq:galphcvg}
			g_\alpha \to \chi_U \text{ in } L^1_{\text{loc} }(U)
		\end{equation}
		because $g_\alpha$ is bounded in $L^\i(U)$.	Given any net $\e_\alpha > 0$ such that $\lim_\alpha \e_\alpha = 0$, the function $g_\alpha$ may be regularized by convolution smoothing using that $V_\alpha \subset \subset U$ such that the regularization $h_\alpha \in C^\i(U; [0, 1] )$ still has compact support in $U$ and satisfies
		\begin{equation}
			\| \nabla h_\alpha \|_{L^1} \le \theta^{-1} C
		\end{equation}
		by \eqref{eq:perimeterctrl} and $\| h_\alpha - g_\alpha \|_{L^1} \le \e_\alpha$. Therefore
		\begin{equation} \label{eq:halphcvg}
			h_\alpha \to \chi_U \text{ in } L^1_{\text{loc} }(U)
		\end{equation}
		by \eqref{eq:galphcvg}. The co-area formula together with the Sard lemma implies
		\begin{equation} \label{eq:perimeterctrl2}
			\begin{gathered}
				\forall \theta \in (0, 1) \, \forall \alpha \, \exists s_\alpha \in (0, \theta) \colon s_\alpha \text{ is a regular value of } h_\alpha, \\
				P \left( \left\{ x \in U \colon h_\alpha(x) > s_\alpha \right\}, U \right) \le \theta^{-1} C.
			\end{gathered}
		\end{equation}
		Setting $U_\alpha \coloneqq \left\{ x \in U \colon h_\alpha(x) > s_\alpha \right\}$ and arguing as for \eqref{eq:galphcvg} with $\chi_{U_\alpha}$ and $h_\alpha$ taking the roles of $g_\alpha$ and $f_\alpha$ in \eqref{eq:measurectrl}, respectively, shows that
		\begin{equation} \label{eq:cvgL1loc}
			\chi_{U_\alpha} \to \chi_U \text{ in } L^1_{\text{loc} }(U)
		\end{equation}
		by \eqref{eq:halphcvg}. Moreover, $U_\alpha \subset \subset U$ because $h_\alpha$ has compact support in $U$. Combining this with \eqref{eq:perimeterctrl2} and \eqref{eq:cvgL1loc}, all claimed properties of $U_\alpha$ have been proved.
	\end{proof}
	
	\begin{proposition} \label{prop:approx1}
		Let $\left| U \right| < + \i$. If there is a net of smooth sets $U_\alpha \subset \subset U$ satisfying $\liminf_\alpha \| D \chi_{U_\alpha} \|(U) < + \i$ and such that $\chi_{U_\alpha}$ converges to $\chi_U$ in $L^1_{\text{loc} }(U)$, then the set $U$ has finite anexometer.
	\end{proposition}
	
	\begin{proof}
		By $U_\alpha \subset \subset U$ and \cite[§5.1, p. 197, Rem. (iii)]{EG}, there holds
		$$
		\| D \chi_{U_\alpha} \|(U) = \vert \chi_{U_\alpha} \vert_{BF(U) }.
		$$
		The seminorm $\vert \cdot \vert_{BF(U) }$ is lower semicontinuous in $\chi_U$ along the net $\chi_{U_\alpha}$ by \Cref{prop:lscseminorm}. Here, we used that $\chi_U \ge \chi_{U_\alpha}$ for every $\alpha$ so that $\chi_{U_\alpha}$ converges to $\chi_U$ in $L^1(V)$ for every bounded subset $V \subset U$. Consequently, \Cref{lem:seminormid} implies
		$$
		\| \chi_U \|_{BF(U) } = \vert \chi_U \vert_{BF(U) } \le \liminf_\alpha \vert D \chi_{U_\alpha} \vert_{BF(U) } = \liminf_\alpha \| D \chi_{U_\alpha} \|(U) < + \i
		$$
		so that $\chi_U \in \BF(U)$.
	\end{proof}
	
	\begin{proposition} \label{prop:approx3}
		Let $A \subset \R^m$ be a set with $0 < \left| A \right| < + \i$. If there is a net of smooth sets $U_\alpha \subset \subset A$ such that $\chi_{U_\alpha} \to \chi_A$ in $L^1(A)$ and $\sup_\alpha P(U_\alpha) < + \i$, then
		\begin{equation} \label{eq:bndryctrl}
			\HH^{m - 1}(\p A \setminus \mathrm{ext}_* A ) < + \i.
		\end{equation}
		Conversely, if $\HH^{m - 1}(\p A \setminus \mathrm{ext}_* A ) < + \i$, then there is not just a net, but even a sequence of smooth sets $V_n \subset \subset A$ such that $\chi_{V_n} \to \chi_A$ in $L^1(A)$ and $\sup_n P(V_n) < + \i$.
	\end{proposition}
	
	\begin{proof}
		The proof when $A$ is bounded and $U_\alpha$ is a sequence instead of a general net \cite[Thm. 3.1]{CLT} carries over to the present setting, as all essential arguments therein are made for a fixed index $\alpha$.
		\\
		Conversely, let $r > 0$ and set $A_r \coloneqq A \cap B_r$, $B_r \coloneqq B_r(0)$ an open ball.
		\begin{equation} \label{eq:fiperi}
			P(A) \le \HH^{m - 1}(\p^* A) \le \HH^{m - 1}(\p A \setminus \mathrm{ext}_* A ) < + \i
		\end{equation}
		by \cite[eq. 3.62]{AFP} and
		\begin{equation} \label{eq:maggid}
			P \left( A_r \right) = P(A; B_r) + P(B_r; \mathrm{int}_* A) + \HH^{m - 1}( \left\{ \nu_A = \nu_{B_r} \right\} )
		\end{equation}
		by \cite[Thm. 16.3, eq. 16.10]{Ma}. Here, $\nu_A$ denotes the measure theoretic outer normal vector field to $A$. Observe $\mathrm{int}_* A_r = \mathrm{int}_* A \cap B_r$ and $\p A_r \subset \p A \cup \p B_r$.\footnote{Let $T$ be a topological space and $M, N \subset T$. Then $\p ( M \cup N ) \subset \p M \cup \p N$.} Hence, we have
		\begin{equation} \label{eq:bdryinc}
			\p A_r \cap \mathrm{int}_* A_r \subset \left( \p A \cup \p B_r \right) \cap \left( \mathrm{int}_* A \cap B_r \right) \subset \p A \cap \mathrm{int}_* A.
		\end{equation}
		Introducing polar coordinates and invoking \cite[eq. 3.62]{AFP} gives
		$$
		\left| \mathrm{int}_* A \right| = \int_0^\i \HH^{m - 1} \left( \mathrm{int}_* A \cap \p B_r \right) \, \mathrm{d}r = \int_0^\i P(B_r; \mathrm{int}_* A) \, \mathrm{d}r < + \i.
		$$
		In particular, there is a sequence of good radii $r_k > 0$ such that
		$$
		r_k \to + \i \text{ and } P(B_{r_k}; \mathrm{int}_* A) \to 0 \text{ as } k \to + \i.
		$$
		Invoking \cite[Thm. 1.6]{GHL} and using \eqref{eq:maggid}, \eqref{eq:fiperi}, and \eqref{eq:bdryinc} provides for every $r > 0$ a sequence $V^r_n \subset \subset A_r$ such that $\chi_{V^r_n} \to \chi_{A_r}$ in $L^1(A_r)$ and
		\begin{align*}
			\limsup_n P(V^r_n)
			& \le P(A_r) + C_m \HH^{m - 1} \left( \p A_r \cap \mathrm{int}_* A_r \right) \\
			& = P(A; B_r) + P(B_r; \mathrm{int}_* A) + \HH^{m - 1}( \left\{ \nu_{\mathrm{int}_* A} = \nu_{B_r} \right\} ) \\
			& + C_m \HH^{m - 1} \left( \p A_r \cap \mathrm{int}_* A_r \right) \\
			& \le P(A) + P(B_r; \mathrm{int}_* A) + \HH^{m - 1}( \p^* \mathrm{int}_* A ) \\
			& + C_m \HH^{m - 1} \left( \p A \cap \mathrm{int}_* A \right) \\
			& \le C + P(B_r; \mathrm{int}_* A).
		\end{align*}
		Therefore, the quantity $\sup_{k \in \N} \limsup_n P(V^{r_k}_n)$ is uniformly bounded because the radii $r_k$ are good. Combining this with the convergence $\chi_{V^r_n} \to \chi_{A_r}$ in $L^1(A_r)$ for every fixed $r > 0$ yields a sequence $V_n$ such that $\chi_{V_n} \to \chi_A$ in $L^1(A)$ and $\sup_n P(V_n) < + \i$ by setting $V_n \coloneqq V^{r_n}_{a_n}$ for some sequence $a_n > 0$ that increases sufficiently rapidly.
	\end{proof}
	
	\begin{theorem} \label{thm:main}
		Let $\emptyset \ne U \subset \R^m$ be an open set with $\left| U \right| < + \i$. The following are equivalent:
		\begin{enumerate}[(i)]
			
			\item The set $U$ has finite anexometer; \label{it:main0}
			
			\item The set $U$ satisfies the divergence theorem of \cref{lem:muf}; \label{it:main1}
			
			\item $\HH^{m - 1} \left( \p U \setminus \mathrm{\ext}_* U \right) < + \i$; \label{it:main2}
			
			\item There is a sequence of smooth sets $V_n \subset \subset U$ with $\sup_n P(U_n) < + \i$ such that $\chi_{V_n} \to \chi_U$ in $L^1(U)$; \label{it:main3}
			
			\item The same as \cref{it:main3} with a \emph{net} $V_\alpha$ instead of the \emph{sequence} $V_n$; \label{it:main4}
			
			\item There is a sequence of functions $f_n \in C^\i_c(U; [0, 1])$ with $\sup_n \| \nabla f_n \|_{L^1} \le \| \chi_U \|_{BF}$ such that $f_n$ converges to $\chi_U$ uniformly on compact subsets of $U$, and $\nabla f_n$ converges to $\mu_{\chi_U}$ weakly* in $\BD'(U)$; \label{it:main5}
			
			\item There is a net of functions $f_\alpha \in C_0(U; [0, 1])$ with $\sup_n \| \nabla f_\alpha \|_{L^1} \le \| \chi_U \|_{BF}$ such that $f_\alpha$ converges to $\chi_U$ weakly* in $\MM'_{\DIV}(U)$, and $\nabla f_\alpha$ converges to $\mu_{\chi_U}$ weakly* in $\BD'(U)$. \label{it:main6}
			
		\end{enumerate}
	\end{theorem}
	
	\begin{proof}
		$\ref{it:main0} \implies \ref{it:main1}$: by \Cref{lem:muf}. $\ref{it:main0} \impliedby \ref{it:main1}$: Immediate by \Cref{def:bndfluc}.
		$\ref{it:main0} \implies \ref{it:main6}$: by \Cref{thm:innerappro}.
		$\ref{it:main6} \implies \ref{it:main5}$: by \Cref{cor:innerappro2}.
		$\ref{it:main5} \implies \ref{it:main4}$: by \Cref{prop:approx2}.
		$\ref{it:main4} \implies \ref{it:main2}$: by the first part of \Cref{prop:approx3}.
		$\ref{it:main2} \implies \ref{it:main3}$: by the second part of \Cref{prop:approx3}.
		$\ref{it:main3} \implies \ref{it:main0}$: by \Cref{prop:approx1}.
	\end{proof}
	
	\emph{Remarks}:
	
	\begin{enumerate}
		
		\item By \Cref{thm:main}, giving an example of a domain having finite perimeter but infinite anexometer is equivalent to finding an open set $U \subset \R^m$ such that $P(U) < + \i$ but $\HH^{m - 1} \left( \p U \setminus \mathrm{\ext}_* U \right) = + \i$. A set satisfying this condition can be obtained by considering the curve
		$$
		\gamma \colon \left[ 1, \i \right) \to \R^2 \colon t \mapsto t^{-1} \left( \cos(t), \sin(t) \right)
		$$
		and setting $U \coloneqq B_1(0) \setminus \ran \gamma$. Then $\p_* U = \p B_1(0)$ so that $P(U) = \HH^1(\p_* U) = \pi^2$, but $\p U \setminus \ext_* U = \p U = \p B_1(0) \cup \ran \gamma \cup \left\{ 0 \right\}$ with
		$$
		\HH^1(\ran \gamma) = \int_1^\i \| \dot{\gamma}(t) \| \, \mathrm{d}t = + \i.
		$$
		
		\item In contrast to the reduced boundary, which in the terminology of \cite[§3.2.14]{Fe3} is countably $\HH^{m - 1}$ rectifiable in $\R^m$, the relevant boundary in \Cref{thm:main} cannot be such in general: let $K \subset \R^m$ be a compact set that is purely $\HH^{m - 1}$ unrectifiable and has infinite $\HH^{m - 1}$ mass. As the interior of $K$ must be empty by pure unrectifiability, setting $V = B_r(0) \setminus K$ for some open ball satisfying $K \subset B_r(0)$, there results an open set $V$ violating \Cref{thm:main}\ref{it:main2} but whose boundary contains no countably $\HH^{m - 1}$ rectifiable set of infinite $\HH^{m - 1}$ mass. This has several consequences: First, the naive approach suggested by the slit disc example in the introduction, i.e., split the domain into finitely many subdomains having finite perimeter and no inner boundary, then recombine the results, cannot treat all domains covered by \Cref{thm:main}. For the reduced boundary of such a partition will necessarily miss the purely unrectifiable part of the boundary. Second, in general, the Hausdorff measure in \Cref{thm:main}\ref{it:main2} cannot be replaced by one of the more lenient integral geometric measures defined in \cite[§2.10.5(1)]{Fe3} due to \cite[Thm. 3.3.13, Cor. 2.10.48]{Fe3}.
		
		\item It is interesting to compare \Cref{thm:main} with \cite[Thm. 3.18]{SS}. Due to the Hahn-Banach norm preserving extension result, as explained at the end of \Cref{ssec:BF1}, the formulation of the divergence theorem used there is seen to be equivalent to mine, at least for the class of bounded sets with finite perimeter treated in \cite[Thm. 3.18]{SS}. Working under a more demanding assumption, \cite[Thm. 3.18]{SS} requires the existence of $\delta_0 > 0$ such that, setting $U_{-\delta} = \left\{ x \in U \colon \dist(x, \p U) > \delta \right\}$, there holds the uniform estimate
		\begin{equation} \label{eq:unifest}
			\sup_{\delta \in (0, \delta_0) } \HH^{m - 1}( \p U_{-\delta} ) < \i,
		\end{equation}
		which implies all of the equivalent conditions given in \Cref{thm:main} by merit of being sufficient for concluding a divergence theorem that implies the underlying domain to have finite anexometer. However, conversely, there is a set of finite anexometer which violates \eqref{eq:unifest}, as can be seen by \cite[Ex. 1]{Kr}. There, an open set $\Omega_0 \subset [ 0, 2] \times [0, 1] \subset \R^2$ of finite perimeter is constructed as an infinite union of disjoint balls. While this rules out the possibility of $\Omega_0$ having a differentiable boundary, because a bounded domain in $\R^m$ cannot have infinitely many connected components if its boundary is differentiable, the set $\p \Omega_0 \setminus \p^* \Omega_0$ may be seen to coincide with $\left\{ 2 \right\} \times [0, 1]$ so that
		$$
		\HH^1(\p \Omega_0) = \HH^1(\p^* \Omega_0) + \HH^1(\p \Omega_0 \setminus \p^* \Omega_0) = P(\p \Omega_0) + 1 < \i.
		$$
		Moreover, defining $\Omega_{0, \delta} = \left\{ x \in \R^2 \colon d(x, \Omega_0) < \delta \right\}$ to be a $\delta$-tube around $\Omega_0$, it its shown in \cite{Kr} that
		$$
		\sup_{\delta \in (0, \delta_0) } \HH^1(\p \Omega_{0, \delta} ) = \i \quad \forall \delta_0 > 0.
		$$
		Thus, setting $\Omega_1 = \interior \left( \R^2 \setminus \Omega_0 \right)$ and using that $\p \Omega_1 = \p \Omega_0$ by the particular structure of $\Omega_0$, we obtain a planar open set such that
		\begin{gather*}
			\HH^1(\p \Omega_1) = \HH^1(\p \Omega_0) < \i, \\
			\sup_{\delta \in (0, \delta_0) } \HH^1(\p \Omega_{1, - \delta} ) = \sup_{\delta \in (0, \delta_0) } \HH^1(\p \Omega_{0, \delta} ) = \i \quad \forall \delta_0 > 0.
		\end{gather*}
		Consequently, the set $\Omega_1 \times (0, 1)^{m - 2}$ is a counterexample ruling out the possibility that every open domain $U \subset \R^m$ of finite anexometer might actually satisfy the stronger uniform estimate \eqref{eq:unifest}.
		
	\end{enumerate}
	
	\newpage
	
	\appendix
	
	\section{Generalized Goldstine theorem}
	
	The next result generalizes the Goldstine theorem \cite[Satz VIII.3.17]{We} in a natural way. It seems absent from the literature. Its versatility is illustrated by the crucial role it has in proving the central approximation result \Cref{thm:innerappro}.
	
	\begin{theorem} \label{thm:predualdens}
		Let $X$ be a Banach space and $K \subseteq X'$ be a weak* closed convex set containing the origin. The predual polar
		$$
		K_\circ = \left\{ x \in X \colon \langle x', x \rangle_{X', X} \le 1 \quad \forall x' \in K \right\}
		$$
		is weak* dense in the bidual polar
		$$
		K^\circ = \left\{ x'' \in X'' \colon \langle x'', x' \rangle_{x'', X'} \le 1 \quad \forall x' \in K \right\}.
		$$
	\end{theorem}
	
	\begin{proof}
		The weak* closure $C$ of $K_\circ$ is contained in $K^\circ$ since $K^\circ$ is weak* closed. Arguing by contradiction, suppose there were $x'' \in K^\circ \setminus C$. Applying the Hahn-Banach theorem \cite[Thm. VIII.2.12]{We} to the dual pair $\left( X', X'' \right)$ gives us $x' \in X'$ such that
		\begin{equation} \label{eq:contradic}
			\sup_{y'' \in C} \langle y'', x' \rangle \le 1 < \langle x'', x' \rangle.
		\end{equation}
		In particular, $\langle x', x \rangle \le 1$ for all $x \in K_\circ$ so that $x' \in K$ by the bipolar theorem \cite[Satz VIII.3.9]{We} as $K$ is a weak* closed convex set containing the origin. But then $\langle x'', x' \rangle \le 1$ for all $x'' \in K^\circ$, which contradicts \eqref{eq:contradic}.
	\end{proof}
	
	\section{A dual space decomposition}
	
	This subappendix provides a result to decompose the dual of an $L^\i$-like function space into a direct topological sum that follows the limit of a given net of measurable subsets decomposing the underlying space. The limiting processes enables to isolate the mass that a finitely additive measure in the dual carries outside the underlying space, e.g., at infinity in the Euclidean setting.
	\\
	This result is formulated in greater generality than is strictly necessary for my application to $\BD(U)$ because abstraction simplifies the exposition. For example, \Cref{cor:DBdeco} carries over from $\BD(U)$ to $C_b(U)$ or $L^\i(U)$ without an essential change.
	\\
	The following notation is in force throughout this subappendix: Let $\left( \Omega, \mathcal{A}, \mu \right)$ be a measure space and $X$ be a real locally convex vector space. Let $L$ be a locally convex vector space of $\mathcal{A}$-measurable functions $u \colon \Omega \to X$ such that
	\begin{subequations}
			\begin{equation}
				\text{The space $L$ is barreled, e.g., Banach or Fréchet;} \label{eq:mondeco3}
			\end{equation}
			\begin{equation} 
				\text{For all $A \in \AA$, the map $L \to L \colon u \mapsto \chi_A u$ is defined, continuous;} \label{eq:mondeco1}
			\end{equation}
			\begin{equation} 
				\text{For all $u \in L$, the family $\chi_A u$ is bounded in $L$ as $A$ ranges over $\AA$.} \label{eq:mondeco2}
			\end{equation}
	\end{subequations}
	The (continuous) dual $L'$ need not be a function space, as it happens if $L = L^\i(0, 1)$. By definition as a space of linear functionals, the dual space $L'$ is separated by $L$, namely
	\begin{equation} \label{eq:sepaprop}
		\forall u' \in L' \setminus \left\{ 0 \right\} \, \exists u \in L \colon \langle u', u \rangle \ne 0.
	\end{equation}
	
	Here and in the following $\langle \cdot, \cdot \rangle$ refers to the dual pairing of $(L, L')$. The idea of proof in \Cref{prop:mondeco} is to study an arbitrary linear functional $\ell$ on a space of measurable functions by considering the auxiliary set function $\AA \to \R \colon A \mapsto \ell \left( \chi_A u \right)$ for a fixed element $u$. This goes back to Giner \cite[Ch. II]{Gn}, who employed it in the context of duality theory. Cf. also \cite[§3.5]{Ru}.
	
	\begin{proposition} \label{prop:mondeco}
		Let $\Omega_\alpha \in \mathcal{A}$ be an isotone (or antitone) net. The multiplication operators $M_\alpha \colon L \to L \colon u \mapsto \chi_{\Omega_\alpha} u$ are linear and continuous. Their adjoint operators $P_\alpha \colon L' \to L'$ converge pointwise in the weak topology $\sigma \left( L', L \right)$ to a linear and continuous projector $P \colon L' \to L'$. If $\| \cdot \|_L$ is a seminorm on $L$ satisfying the implication
		\begin{equation} \label{eq:normimp}
			\max\{ \| u \|_L, \| v \|_L \} \le 1 \land A \in \mathcal{A} \implies \| \chi_A u + \chi_{\Omega \setminus A} v \|_L \le 1,
		\end{equation}
		and $\| u' \|_{L'} \coloneqq \sup_{\| u \|_L \le 1} \langle u', u \rangle$ denotes the dual seminorm, then
		$$
		\| u' \|_{L'} = \| P u' \|_{L'} + \| \left( 1 - P \right) u' \|_{L'} \quad \forall u' \in L'.
		$$
	\end{proposition}
	
	\begin{proof}
		Each $M_\alpha$ and hence every $P_\alpha$ defines a linear and continuous operator by \eqref{eq:mondeco1}. Clearly, $M^2_\alpha = M_\alpha$, i.e., $M_\alpha$ projects and hence so does $P_\alpha$. Fix $u \in L$ and $u' \in L'$. \Cref{eq:mondeco1} allows to define the finitely additive set function $\nu \colon \mathcal{A} \to \R \colon A \mapsto \langle u', \chi_A u \rangle$, which is finite since the continuous functional $u'$ is bounded and \eqref{eq:mondeco2} holds. Consequently, $\nu$ has a positive part $\nu^+$ and a negative part $\nu^-$. If the net $\Omega_\alpha$ is isotone, then
		$$
		\nu^+ \left( \Omega_\alpha \right) \le \nu^+ \left( \Omega_\beta \right) \le \nu^+ \left( \Omega \right) < \i \quad \forall \alpha \le \beta
		$$
		so that $\lim_\alpha \nu^+ \left( \Omega_\alpha \right) = \sup_\alpha \nu^+ \left( \Omega_\alpha \right)$ and $\lim_\alpha \nu^- \left( \Omega_\alpha \right) = \sup_\alpha \nu^- \left( \Omega_\alpha \right)$. If $\Omega_\alpha$ is antitone, the suprema are to be replaced by infima. Putting these observations together obtains the pointwise convergence
		$$
		\lim_\alpha \langle P_\alpha u', u \rangle = \lim_\alpha \nu \left( \Omega_\alpha \right) = \lim_\alpha \nu^+ \left( \Omega_\alpha \right) - \lim_\alpha \nu^- \left( \Omega_\alpha \right) \quad \forall (u, u') \in L \times L'.
		$$
		The linear limit operator $P$ is a continuous by \cite[Cor. 7.1.4]{Ed} and \eqref{eq:mondeco3}. To see that $P$ projects, observe that for every $\alpha$,
		$$
		\chi_{\Omega_\alpha} u = 0 \implies \langle P_\beta u', u \rangle = 0 \quad \forall \beta \ge \alpha \implies \langle P u', u \rangle = 0.
		$$
		Therefore, as $\chi_{\Omega \setminus \Omega_\alpha} u \in L$ by \eqref{eq:mondeco1},
		$$
		\langle P u', u \rangle = \langle P u', \chi_{\Omega_\alpha} u \rangle = \langle P_\alpha P u', u \rangle \to \langle P^2 u', u \rangle
		$$
		as $\alpha$ surpasses any bound. The element $u$ being arbitrary, the functionals $Pu'$ and $P^2 u'$ define the same element of $L'$ for any $u'$. Addendum: for every $\e > 0$, there are $u, v \in L$ such that $\max\{ \| u \|, \| v \| \} \le 1$ and
		\begin{align*}
			\| u' \|_{L'}
			\le \| P u' \|_{L'} + \| \left( 1 - P \right) u' \|_{L'}
			& \le \langle Pu ', u \rangle + \langle \left( 1 - P \right) u', v \rangle + \e \\
			& = \lim_\alpha \langle u', \chi_{\Omega_\alpha} u + \chi_{\Omega \setminus \Omega_\alpha} v \rangle + \e \\
			& \le \| u' \|_{L'} + \e
		\end{align*}
		because $\| \chi_{\Omega_\alpha} u + \chi_{\Omega \setminus \Omega_\alpha} v \|_L \le 1$ by assumption.
	\end{proof}
	
	\begin{proposition} \label{prop:unideco}
		In the setting of \Cref{prop:mondeco}, if two isotone (or antitone) nets $\Omega_\alpha$ and $\Omega'_{\alpha'}$ are eventually contained in each other, meaning
		\begin{equation} \label{eq:evcontain}
			\begin{aligned}
				& \forall \alpha \, \exists \alpha'_0 \colon \alpha' \ge \alpha'_0 \implies \Omega_\alpha \subseteq \Omega'_{\alpha'}, \\
				& \forall \alpha' \, \exists \alpha_0 \colon \alpha \ge \alpha_0 \implies \Omega'_{\alpha'} \subseteq \Omega_\alpha,
			\end{aligned}
		\end{equation}
		(reverse inclusions for antitone nets), then $\lim_\alpha P_\alpha = \lim_{\alpha'} P_{\alpha'}$. In particular, if the net is isotone, each $\Omega_\alpha$ belongs to some system of measurable sets $\mathfrak{S} \subset \mathcal{A}$ and each set $S \in \mathfrak{S}$ is eventually contained in $\Omega_\alpha$, then the limit projector $P$ from \Cref{prop:mondeco} does not depend on the particular choice of such a net.
	\end{proposition}
	
	\begin{proof}
		The notation from \Cref{prop:mondeco} remains in force. \Cref{eq:evcontain} implies
		$$
		\forall \alpha \, \exists \alpha'_0 \colon \alpha' \ge \alpha'_0 \implies \nu^+ \left( \Omega_\alpha \right) \le \nu^+ \left( \Omega'_{\alpha'} \right) 
		$$
		for isotone nets $\Omega_\alpha$ and $\Omega'_{\alpha'}$ so that $\lim \nu^+ \left( \Omega_\alpha \right) \le \lim \nu^+ \left( \Omega'_{\alpha'} \right)$.	Switching the roles of $\Omega_\alpha$ and $\Omega'_{\alpha'}$, one analogously deduces $\lim \nu^- \left( \Omega_\alpha \right) \ge \lim \nu^- \left( \Omega'_{\alpha'} \right)$. Hence, in total,
		\begin{equation} \label{eq:nuineq}
			\lim \nu \left( \Omega_\alpha \right) \le \lim \nu \left( \Omega'_{\alpha'} \right).
		\end{equation}
		As $\Omega_\alpha$ and $\Omega_{\alpha'}$ have symmetric roles, equality follows in \eqref{eq:nuineq}. Consequently,
		$$
		\lim \langle P_\alpha u', u \rangle = \lim \langle P_{\alpha'} u', u \rangle \quad \forall (u, u') \in L \times L'
		$$
		so that $\lim P_\alpha = \lim P_{\alpha'}$ by \eqref{eq:sepaprop}. For antitone nets, reverse the inequalities. Addendum: if $\Omega_\alpha$ and $\Omega'_{\alpha'}$ are such nets, then $\Omega_\alpha \subset \Omega'_{\alpha'}$ eventually and vice versa so that the first part applies.
	\end{proof}
	
	\begin{lemma} \label{lem:unifdeco}
		Let $\mathfrak{S} \subset \mathcal{A}$ be a system of measurable sets and $\Omega_\alpha \in \mathfrak{S}$ be an isotone net that eventually contains any given $S \in \mathfrak{S}$. Moreover, let $\varphi_{\alpha'} \colon \Omega \to \left[ 0, 1 \right]$ be a net of measurable functions such that $\varphi_{\alpha'} \to 1$ uniformly on every $S \in \mathfrak{S}$ and $\left\{ \varphi_{\alpha'} > 0 \right\} \in \mathfrak{S}$. If $u \in L \implies \varphi_{\alpha'} u \in L$, then the net of operators $M'_{\alpha'} \colon L' \to L'$ defined as the adjoint of the multiplication operators $M_{\alpha'} \colon L \to L \colon u \mapsto \varphi_{\alpha'} u$ converges pointwise in the weak topology $\sigma \left( L', L \right)$ to the projector $P \colon L' \to L'$ associated with the net $\Omega_\alpha$ by \Cref{prop:mondeco}.
	\end{lemma}
	
	\begin{proof}
		The notation from \Cref{prop:mondeco} remains in force. Let $\nu^+ \left( \varphi_{\alpha'} \right)$ denote the integral of $\varphi_{\alpha'}$ with respect to $\nu^+$. For any indices $\alpha$ and $\alpha'$,
		$$
		\nu^+ \left( \varphi_{\alpha'} \right)
		\le	\nu^+ \left( \Omega_\alpha \right) + \int_{\Omega \setminus \Omega_\alpha} \varphi_{\alpha'} \, \mathrm{d} \nu^+.
		$$
		As $\left\{ \varphi_{\alpha'} > 0 \right\} \in \mathfrak{S}$ for every fixed $\alpha'$, this set is eventually contained in $\Omega_\alpha$ so that taking the limit first in $\alpha$ using $\varphi_{\alpha'} \ge 0$ and then in $\alpha'$ gives
		\begin{equation} \label{eq:ineq1}
			\limsup \nu^+ \left( \varphi_{\alpha'} \right)
			\le \liminf \nu^+ \left( \Omega_\alpha \right).
		\end{equation}
		Also, as $\varphi_{\alpha'} \le 1$ and $\varphi_{\alpha'} \to 1$ uniformly on $\Omega_\alpha \in \mathfrak{S}$ for every fixed $\alpha$,
		$$
		\nu^+ \left( \Omega_\alpha \right)
		= \lim \int_{\Omega_\alpha} 1 - \varphi_{\alpha'} \, \mathrm{d} \nu^+
		+ \lim \int_{\Omega_\alpha} \varphi_{\alpha'} \, \mathrm{d} \nu^+
		\le \liminf \nu^+ \left( \varphi_{\alpha'} \right).
		$$
		Taking the limit in $\alpha$ gives
		\begin{equation} \label{eq:ineq2}
			\limsup \nu^+ \left( \Omega_\alpha \right)
			\le \liminf \nu^+ \left( \varphi_{\alpha'} \right).
		\end{equation}
		Putting \eqref{eq:ineq1} and \eqref{eq:ineq2} together shows that $\lim \nu^+ \left( \Omega_\alpha \right) = \lim \nu^+ \left( \varphi_{\alpha'} \right)$. Arguing analogously, one obtains the same for $\nu^-$ instead of $\nu^+$. Consequently,
		$$
		\lim \langle M'_{\alpha'} u', u \rangle = \lim \langle P_\alpha u', u \rangle = \langle P u', u \rangle \quad \forall u', u
		$$
		so that \eqref{eq:sepaprop} implies $\lim M'_{\alpha'} = P$.
	\end{proof}
	
	\begin{theorem} \label{thm:duadeco}
		Let $Y$ be a Banach space and $\Lambda \left( \mu ; Y \right)$ be the Banach space of all (equivalence classes of) measurable functions $u \colon \Omega \to Y$ that are $\mu$-a.e. equal to a bounded measurable function equipped with the essential supremum norm
		$$
		\| u \|_\i = \inf \left\{ \alpha > 0 \mid \| u(\omega) \|_Y \le \alpha \text{ for $\mu$-a.e. } \omega \in \Omega \right\}.
		$$
		Moreover, let $V \subset \Lambda \left( \mu ; Y \right)$ be a closed subspace, $\mathfrak{S} \subset \mathcal{A}$ be a system of measurable sets and $\Omega_\alpha \in \mathfrak{S}$ be an isotone net such that $S \subset \Omega_\alpha$ eventually for every $S \in \mathfrak{S}$. Let $V_0$ be the closure of functions $u \in V$ such that $\left\{ u \ne 0 \right\} \in \mathfrak{S}$ and $\varphi_{\alpha'} \colon \Omega \to \left[ 0, 1 \right]$ be a net of measurable functions inducing multiplication operators on $\Lambda \left( \mu ; Y \right)$ as in \Cref{lem:unifdeco} with the additional property that $u \in V \implies \varphi_{\alpha'} u \in V$.
		\\
		
		If $P \colon \Lambda \left( \mu ; Y \right)' \to \Lambda \left( \mu ; Y \right)'$ is the projector induced by the net $\Omega_\alpha$ according to \Cref{prop:mondeco}, then
		$$
		P_V \ell' = P \lambda' + V^\circ \text{ if } \ell' = \left[ \lambda' \right] = \lambda' + V^\circ,
		$$
		is a well-defined, linear and continuous projector $P_V \colon V' \to V'$. Moreover,
		\begin{equation} \label{eq:isometrisom}
			\ran(P_V) \cong V' / V_0^\circ \cong V'_0, \quad
			\ker(P_V) = V_0^\circ, \quad
			V' = \ran(P_V) \oplus_1 \ker(P_V),
		\end{equation}
		where $P_V$ induces an isometric isomorphism $\ran(P_V) \cong V'_0$ by unique norm preserving extension from $V_0$ to $V$ as an element of $\ran(P_V)$.
	\end{theorem}
	
	\begin{proof}
		Let $M_{\alpha'} \colon \Lambda \left( \mu; Y \right) \to \Lambda \left( \mu; Y \right) \colon u \to \varphi_{\alpha'} u$. Setting $P_{V, \alpha'}$ to be the adjoint of $\left. M_{\alpha'} \right|_V$ defines a linear, continuous operator. Since $\left. M_{\alpha'} \right|_V$ projects, $P_{V, \alpha'}$ also does. Moreover, as the adjoint $P_{\alpha'}$ converges pointwise in the weak* sense, so does $P_{V, \alpha'}$ and $P_V = \lim P_{V, \alpha'}$. It remains to prove the addenda.
		\\
		
		To prove the first of \eqref{eq:isometrisom}, if suffices to show that $\ran(P_V) \cong V' / V_0^\circ$ because $V' / V_0^\circ \cong V'_0$ is always true by \cite[Satz 3.1.10]{We}. As $\ker(P_V) = V_0^\circ$, we have a well-defined, linear, and continuous mapping
		$$
		Q_V \colon V' / V_0^\circ \to \ran(P_V) \colon u' + V_0^\circ \to P_V u'.
		$$
		The mapping is bijective by definition. Moreover, $\| Q_V \| = \| P_V \| \le \| P \| \le \liminf \| P_{\alpha'} \| \le 1$ shows that $Q_V$ is non-expansive. If $v' = P_V u'$, then
		$$
		Q_V \left( P_V u' + V_0^\circ \right) = P_V^2 u' + P_V V_0^\circ = P_V u' = v'
		$$
		because $P_V$ projects. Combining this with $\| P_V u' + V_0^\circ \|_{V' / V_0^\circ} \le \| P_V u' \|_{V'}$ demonstrates that $Q_V$ maps the unit ball onto the unit ball. In total, one concludes that it is the required isometric isomorphism.
		\\
		
		Regarding the second of \eqref{eq:isometrisom}, let $u' \in V'$ with $u' = \left[ \ell' \right]$ for $\ell' \in \Lambda \left( \mu ; Y \right)'$ and $u \in V$ with $\left\{ u \ne 0 \right\} \in \mathfrak{S}$. Then
		\begin{equation} \label{eq:annihi}
			\begin{aligned}
				\langle P_V u', u \rangle_{V', V}
				& = \langle P \ell', u \rangle_{\Lambda', \Lambda}
				= \lim_\alpha \langle \ell', \chi_{\Omega_\alpha} u \rangle_{\Lambda', \Lambda} \\
				& = \langle \ell', u \rangle_{\Lambda', \Lambda}
				= \langle u', u \rangle_{V', V}.
			\end{aligned}
		\end{equation}
		Therefore $\left( 1 - P_V \right) u' \in L^\circ$ so that $\ran \left( 1 - P_V \right) = \ker \left( P_V \right) \subset V_0^\circ$. Conversely, if $u' \in \ker \left( P_V \right)$, then \eqref{eq:annihi} shows $0 = \langle u', u \rangle_{V', V}$ for every $u \in V_0$ so that $V_0^\circ \subset \ker \left( P_V \right)$.
		\\
		
		The third of \eqref{eq:isometrisom} follows similarly to the corresponding direct sum decomposition in \Cref{prop:mondeco} because now, in analogy with \eqref{eq:normimp},
		\begin{equation*}
			\max\{ \| u \|_\i, \| v \|_\i \} \le 1 \implies \| \varphi_{\alpha'} u + \left( 1 - \varphi_{\alpha'} \right) v \|_\i \le 1 \quad \forall u, v \in V.
		\end{equation*}
		Finally, to see that $\ran(P_V) \cong V'_0$ by means of unique norm preserving extension, note that extending from $V_0$ to $\ran(P_V)$ is unique since $Q_V$ is an isomorphism and the equivalence class of any extension is mapped into the extension by $Q_V$ since $P_V$ projects. Moreover, no further norm preserving extension outside of $\ran(P_V)$ can exist because if $v' \in V'$ is an extension, then $\| v' \| = \| P_V v' \| + \| (1 - P_V) v' \| = \| v' \| + \| (1 - P_V) v' \|$ so that $v' = P_V v'$ belongs to $\ran(P_V)$.
	\end{proof}
	
	\begin{corollary} \label{cor:DBdeco}
		Let $U \subset \R^m$ be open and
		\begin{alignat*}{2}
			\BD_0(U)   & \coloneqq \cl \left\{ u \in \BD(U) \colon \supp(u) \subset\subset U \right\}, \\[0.3em]
			\BD_b(U)   & \coloneqq \cl \left\{ u \in \BD(U) \colon \supp(u) \text{ is bounded } \right\}, \\[0.3em]
			\BD'_\p(U) & \coloneqq \BD_0(U)^\circ \text{ in } \BD'_b(U), \\[0.3em]
			\BD'_\i(U) & \coloneqq \BD_b(U)^\circ \text{ in } \BD'(U).
		\end{alignat*}
		Then $\BD'_0(U)$ is a subspace of $\BD'_b(U)$ and $\BD'_b(U)$ is a subspace of $\BD'(U)$ by unique norm preserving extension and there hold the isometric isomorphisms
		$$
		\BD'_b(U) \cong \BD'_0(U) \oplus_1 \BD'_\p(U), \quad
		\BD'(U) \cong \BD'_b(U) \oplus_1 \BD'_\i(U).
		$$
		Moreover, the projection operator $P \colon \BD'_b(U) \to \BD'_0(U)$ can be obtained by choosing in \Cref{lem:unifdeco} the system of sets that are compactly contained in $U$. Similarly, the projection operator $P \colon \BD'(U) \to \BD'_b(U)$ can be obtained by choosing in \Cref{lem:unifdeco} the system of bounded sets contained in $U$.
	\end{corollary}
	
	\begin{proof}
		It suffices to check the addenda by \Cref{thm:duadeco}. First, to check $\BD'_b(U) \cong \BD'_0(U) \oplus_1 \BD'_\p(U)$, let $V = \BD_b(U)$, take $\mathfrak{S}$ as the system of sets that are compactly contained in $U$, and define
		\begin{gather*}
			U_n = \left\{ x \in U \colon \dist(x, U^c) > \tfrac{1}{n} \right\}, \\
			\varphi_n(x) = \left( 1 - 2 n \dist \left( x, U_n \right) \right)^+, \quad
			\Omega_n = \left\{ \varphi_n > \tfrac{1}{2} \right\}.
		\end{gather*}
		The non-decreasing sequence of open sets $\Omega_n$ converges to $U$ so that every relatively compact subset of $U$ is eventually contained in $\Omega_n$. Moreover, $\BD_0(U)$ is the closure of functions $u$ for which $\left\{ u \ne 0 \right\} \in \mathfrak{S}$ while the boundedness and Lipschitz continuity of $\varphi_n$ together with a product rule for elements of $\DM(U)$ proves that
		\begin{equation} \label{eq:sp inc}
			\varphi_n \BD(U) \subset \BD(U).
		\end{equation}
		Next, we claim that $\BD'(U) \cong \BD'_b(U) \oplus_1 \BD'_\i(U)$. To check this, let $V = \BD(U)$, take $\mathfrak{S}$ as the system of bounded sets in $U$, and define
		$$
		B_n = B_n(0), \quad \varphi_n(x) = \left( 1 - n \dist(x, B_n ) \right)^+, \quad \Omega_n = U \cap \left\{ \varphi_n > \tfrac{1}{2} \right\}.
		$$
		Then $\Omega_n$ is an isotone sequence of opens sets that eventually contains any $B_n \cap U$, hence contains any given bounded subset of $U$. Moreover, $\BD_b(U)$ is the closure of functions $u$ for which $\left\{ u \ne 0 \right\} \in \mathfrak{S}$ while \eqref{eq:sp inc} for the same reasons as before.
	\end{proof}
	
%	\section{Data availability statement}
%	
%	I do not analyse or generate any datasets, because my work proceeds within a theoretical and mathematical approach.

\end{document}